\documentclass{amsart}

\usepackage{amsmath,amsthm,amscd,amssymb}
\usepackage{mathrsfs}
\usepackage{graphicx}
\usepackage{enumerate}

\usepackage{ascmac}
\usepackage{color}
\usepackage{here}

\theoremstyle{plain}

\newtheorem{thm}{Theorem}[section]
\newtheorem{cor}[thm]{Corollary}
\newtheorem{lem}[thm]{Lemma}
\newtheorem{prop}[thm]{Proposition}

\theoremstyle{definition}
\newtheorem{dfn}[thm]{Definition}

\theoremstyle{remark}
\newtheorem{remark}[thm]{Remark}

\numberwithin{equation}{section}
\newtheorem*{rmk}{Remark}

\newcounter{tmp}

\title[
Emergence via  
non-existence of averages]
{
Emergence via  
 non-existence of averages}

\date{\today}

\author{Shin Kiriki}
\address[Shin Kiriki]{Department of Mathematics, Tokai University, 4-1-1 Kitakaname, Hiratuka, Kanagawa, 259-1292, JAPAN}
\email{kiriki@tokai-u.jp}

\author{Yushi Nakano}
\address[Yushi Nakano]{Department of Mathematics, Tokai University, 4-1-1 Kitakaname, Hiratuka, Kanagawa, 259-1292, JAPAN}
\email{yushi.nakano@tsc.u-tokai.ac.jp}

\author{Teruhiko Soma}
\address[Teruhiko Soma]{Department of Mathematical Sciences, Tokyo Metropolitan University, 1-1 Minami-Ohsawa, Hachioji, Tokyo, 192-0397, JAPAN}
\email{tsoma@tmu.ac.jp}



\begin{document}

\begin{abstract}

Inspired by a recent work by  Berger, 
we introduce the concept of pointwise emergence.  
This concept provides with a new quantitative perspective into the study of non-existence of averages
 for dynamical systems. 
We show that high pointwise emergence on a large set appears for abundant dynamical systems:  Any continuous maps on a compact metric space with the specification property have super-polynomial pointwise emergence on a residual subset of the state space. 
 Furthermore, there is a dense subset of any Newhouse open set each element of which has super-polynomial pointwise emergence on a positive Lebesgue measure subset of the state space. \end{abstract}

\maketitle

\section{Introduction}\label{s:intro}

The study of 
  \emph{infinitude} or \emph{non-existence  of averages} for dynamical systems 
  has  a long  history, despite  being  beyond the Smale-Palis program 
  \cite{PS1970,Palis2000,BDV04} 
  which has been a guiding principle in modern dynamical systems theory. 
It is only 70's that Newhouse showed in \cite{Newhouse74} that  there is a residual subset of any Newhouse open set (see Section \ref{s:keydfn} for precise definition) 
  each element 
   of which 
 has 
 infinitely many sinks. 
Furthermore, it is Bowen who first studied  dynamics without  time averages on   a positive Lebesgue measure set (although it was never published by himself, see \cite{Takens1994}). 
On the other hand, it is recent that Berger \cite{Berger2016} proved that typical dynamics (in the sense of Kolmogorov) in any Newhouse open set have infinitely many sinks, and that the first and third authors  \cite{KS2017}  showed  that there is a dense subset of any Newhouse set  each element of which has 
  a positive Lebesgue measure set  where 
 time averages do not exist. 
We refer to \cite{BDV04,Berger2016,KS2017} for detailed history. 

Recently, Berger \cite{Berger2017} introduced a  quantitative viewpoint into the study of  infinitude of averages, and further developed it in \cite{BB} with Bochi. 
In the paper \cite{Berger2017},   a 
``global'' 
$\epsilon$-approximation of empirical measures 
(i.e.~measures representing averages) 
of a dynamical system is called 
 \emph{emergence}  at scale $\epsilon >0$
 (we note that ``emergence'' is 
 one of the most important concepts in  complexity science \cite{MacKay2008}, 
  but had no rigorous formulation before  \cite{Berger2017} appeared),  
 and it is shown 
  that the growth rate of emergence in the limit $\epsilon \to 0$ captures the complexity of a dynamical system with \emph{infinitude of averages}. 
Our purpose in this paper is to investigate  ``local'' emergence (called \emph{pointwise emergence}, Definition \ref{dfn:main}). 
We will   see that pointwise emergence well adapts to the study of complexity of \emph{non-existence of averages}, 
resulting in 
 a strong contrast between   pointwise emergence and Berger's emergence 
 (see Section  \ref{ss:mte}).  
Furthermore, we prove that high pointwise emergence on a large set appears for abundant dynamical systems (Proposition \ref{thm:expanding} and Theorem \ref{thm:main}).

\subsection{Emergences}

We first  briefly recall the definition of 
Berger's emergence. 
Let $X$ be a compact metric space and $f: X\to X$ a continuous map. 
We study  
 empirical measures $\{ \delta _x^n\} _{n\geq 1}$  given by
\[
\delta _x^n \equiv \delta_x^n(f) = \frac{1}{n} \sum _{j=0}^{n-1} \delta _{f^j(x)} \quad (\text{$x\in X$, $n\geq 1$}),
\]
 where $\delta _y$ is the Dirac measure at $y\in X$. 
Note that $\int \varphi  d\delta _x^n =1/n \sum _{j=0}^{n-1} \varphi (f^j(x))$  is the (partial) time average of a continuous function $\varphi $  (along the orbit of $x$ by $f$) at $n\geq 1$, so that the study of  asymptotic behavior of $\{ \delta _x^n\} _{n\geq 1}$ in weak
 topology  would be most fundamental  
 in ergodic theory. 
We metrize the weak topology of 
the space  $\mathcal P(X)$   of probability measures on $X$ 
 by the first Wasserstein metric $ d \equiv W_1$ (see Section \ref{s:keydfn} for  the definition  of $ W_1$): 
recall that 
convergence 
 with respect to $W_1$ is equivalent to  the weak convergence 
  (refer to   e.g.~\cite[Theorems 6.9]{Villani2008}; other classical metrics to metrize the weak topology of $\mathcal P(X)$, such as L\'evy-Prokhorov metirc,  were also considered in \cite{BB}). 
By virtue of (a straightforward modification of) Proposition 1.10 of  \cite{Berger2017},   $x\mapsto  d \left( \delta _x^n , \nu \right)$ is  continuous  for any $n\geq 1$ and $\nu \in \mathcal P(X)$.

Let  $M$ be a compact manifold and $f$ a continuous map on $M$.
 In \cite{Berger2017},  Berger defined 
 the \emph{emergence} $\mathscr E_{\mathrm{Leb}}(\epsilon ) \equiv \mathscr E_{\mathrm{Leb}}(\epsilon ,f)$ of $f$  at scale $\epsilon>0$ by
\begin{multline}\label{dfn:metricemergence}
\mathscr E_{\mathrm{Leb}}(\epsilon ) = \min \Big\{ N \in \mathbb N \mid  \text{there exists   $\{ \mu _j\} _{j=1}^N \subset \mathcal P(M)$ such that} \\
 \limsup _{n\to \infty}\int _M  \min _{1\leq j\leq N}  d \left( \delta _x^n , \mu _j \right)d\mathrm{Leb} (x) \leq \epsilon \Big \} ,
\end{multline}
where $\mathrm{Leb}$ is the normalized Lebesgue measure on $M$. 
This was called \emph{metric emergence} in \cite{BB}, because they needed to distinguish it with  another emergence (called \emph{topological emergence}, see \eqref{eq:0401} for definition). 
We also use their terminology, that is, $\mathscr E_{\mathrm{Leb}}(\epsilon )$ will be called metric emergence at scale $\epsilon >0$. 
(To be more precise, in \cite{BB} they also studied metric emergences $\mathscr E_\mu (\epsilon )$ for any probability measure $\mu$ on a compact metric space $X$ (not  necessarily a manifold) defined by \eqref{dfn:metricemergence} with $\mu$ instead of $\mathrm{Leb}$,  and obtained a variational principle for metric and topological emergences.)

The inequality in \eqref{dfn:metricemergence} means that $\{ \mu _j\}_{j=1}^N$ approximates the  statistics of $f$ in the $\epsilon$ scale. Hence, once one fixes 
 $\epsilon$, the complexity of  statistics of $f$ ``emerges'' as $\mathscr E_{\mathrm{Leb}}(\epsilon)$. 
 Interesting examples are as follows: 
It is shown in \cite[Section 1.2]{Berger2017} that if $f$ has finitely many ergodic probability  measures such that the union of basins of the measures covers $M$ up to a zero Lebesgue measure set, then $\mathscr E_{\mathrm{Leb}}(\epsilon )$ is bounded by the number of the measures for any $\epsilon$. 
 On the other hand,  $\lim _{\epsilon \to 0}\mathscr E_{\mathrm{Leb}}(\epsilon ) = \infty$  if $f$ has infinitely many sinks (\cite[Claim 1.13]{Berger2017}) 
 or if $f$ is a conservative system on the annulus $\mathbb S^1 \times [0,1]$ which preserves each circle $\mathbb S^1 \times \{\rho \}$ with $\rho \in [0,1]$ and satisfies   a very mild condition (\cite[Proposition 4.1]{BB}). 
In this sense, we may say that metric emergence well captures infinitude of averages.

Under the background of naive and massive uses of computer approximation of statistics in many branches of  sciences, Berger started  a program to prove that for each typical dynamics (in the sense of Kolmogorov) in an open set of the space of diffeomorphisms, 
the metric emergence  is \emph{super-polynomial}, that is, 
$\limsup _{\epsilon \to 0} \log \mathscr E_{\mathrm{Leb}}(\epsilon )/(-\log \epsilon ) =\infty$ (or equivalently,   $\limsup _{\epsilon \to 0} \epsilon ^{\alpha}\mathscr E_{\mathrm{Leb}}(\epsilon )  =\infty$ for any $\alpha \geq 0$), 
see  \cite[Problem 1.14]{Berger2017}. 
Among computer scientists, an algorithm of super-polynomial complexity 
 is thought to be \emph{not feasible in practice by a computer} \cite{Cobham1965}, so that the accomplishment of the  program may give
an alarm to  the aforementioned optimistic trend. 
A great contribution to the program was recently made 
 in \cite{BB}. 
We also 
 remember that 
another   \emph{quantitative} study of generic non-hyperbolic dynamics  
  by Kaloshin \cite{Kaloshin2000}
(i.e.~super-exponential growth of number of periodic orbits for generic dynamics in Newhouse open sets, in which  infinitely many sinks  exist for generic dynamics)
 opened up a fruitful research field.

A feature of the metric emergence is the \emph{integration} in \eqref{dfn:metricemergence} resulting in a grasp of ``global'' statistical information of the dynamics. 
In this paper, we consider 
 ``local'' emergence as follows.  
 Let $X$ be a compact metric space (not necessarily a manifold).
 \begin{dfn}\label{dfn:main}
Given $\epsilon >0$ and $x\in X$,  the \emph{pointwise emergence} $\mathscr E_x(\epsilon ) \equiv \mathscr E_x(\epsilon ,f)$ of $f$  at scale $\epsilon$ at $x$ is defined by
\begin{multline}\label{eq:0403d}
\mathscr E_x(\epsilon ) = \min \Big\{ N \in \mathbb N \mid  \text{there exists   $\{ \mu _j\} _{j=1}^N \subset \mathcal P(X)$ such that} \\
\limsup _{n\to \infty}  \min _{1\leq j\leq N}  d \left( \delta _x^n , \mu _j \right) \leq \epsilon \Big \}.
\end{multline}
The pointwise emergence at $x\in X$ is called \emph{super-polynomial} 
 if 
\[
\limsup _{\epsilon \to 0}  \frac{\log \mathscr E_x(\epsilon ) }{-\log \epsilon } =\infty.
\] 
\end{dfn}

\subsection{Historic behavior}
\label{subsection:hb}

We can see that the pointwise emergence gives a quantitative perspective into  non-existence of averages, or historic behavior. 
Recall that a point $x\in X$ (or its forward orbit) is said to have \emph{historic behavior} if the time average 
$
\lim _{n\to \infty} \delta _x^n
$ 
does not exist. 
(This terminology originates from Ruelle \cite{Ruelle2001}; see also \cite{Takens2008}.) 
Although the set of points with historic behavior is a $\mu$-zero measure set for any invariant measure $\mu$ due to Birkhoff's ergodic theorem (so that the set   is  called the \emph{irregular set} or the \emph{non-typical set} in the context of thermodynamic formalism \cite{BS2000, Thompson2012}), 
the  set is known to be remarkably large for many dynamical systems. 

Known dynamical systems  with historic behavior on a measure-theoretically large set are as follows. 
It is a famous folklore that Bowen knew that a surface flow with heteroclinically connected two dissipative saddle points 
 has a positive Lebesgue measure set consisting of 
  points with historic behavior (see \cite{Gaunersdorfer1992,Takens1994} for precise proof). 
  We emphasize that  for Bowen's example, there are many ``abnormal'' results 
  other than 
    historic behavior, 
 refer to  e.g.~\cite{BBS1999,Araujo2000,Araujo2001,OY2008,AP2018}. 
However, Bowen's example is easily broken by small perturbations, and thus
Takens asked in \cite{Takens2008} whether there is a persistent class of diffeomorphisms
for which the set of points with historic behavior is of positive Lebesgue measure
(called \emph{Takens' Last Problem}). 
The first and third authors \cite{KS2017} affirmatively answered it by showing that 
there is a  dense subset of any Newhouse open set  in the set of $\mathcal C^r$ surface diffeomorphisms ($2\leq r<\infty$) such that any element of the dense set has 
 a wandering domain consisting of points with historic behavior, 
by employing the best technology developed by Colli-Vargas \cite{CV2001}  for wandering domains near homoclinic tangency.  
 Very recently, Berger and  Biebler 
extended it to the $\mathcal C^\infty$ and  analytic cases (\cite{BB2020}). 

One can also find other
interesting examples with a positive Lebesgue measure set consisting of 
  points with historic behavior 
for some quadratic maps 
 in \cite{HK1990}, 
   for  flows generated by   3-dimensional vector fields (in  a locally dense set) with heteroclinic cycles between periodic solutions 
    in \cite{LR2016}, 
and  for some partially hyperbolic dynamics which is a compactification  of 
an $\mathbb R$-extension of an Anosov diffeomorphism 
sharing properties with the Brownian motion on $\mathbb R$ in \cite{CYZ2018}.

From topological  viewpoint, we can  find more examples with historic behavior on a large set. 
  Sigmund essentially showed in \cite{Sigmund1974}  that any continuous map   on a compact metric space with  the specification property has a residual subset of the state space consisting of points with historic behavior (this is explicitly stated and proven by himself in \cite[Proposition 21.18]{DGS1976}). 
See Subsection \ref{subsection:01202} 
for the definition 
 of 
the specification property.
We here just remember that any topologically mixing subshift of finite type
 satisfies the specification property (\cite[Proposition 21.2]{DGS1976}).
This result was extended to 
 shifts with weak specification in   \cite{BLV2014},   geometric Lorenz flows in \cite{KLS2016}, sectionally hyperbolic flows in \cite{AP2018}, and  
  $\mathcal C^1$-generic diffeomorphisms with non-hyperbolic homoclinic classes in \cite{BKNSR}. 
 In  the context of thermodynamic formalism,  a large contribution  to historic behavior 
  was also made 
 by several authors. 
A very incomplete list of them is \cite{PP84, BS2000, CKS2005, Thompson2012, CZZ2011, CTV2015, BV2017, BLV2018}. 
We here merely mention that Pesin and Pitskel' \cite{PP84}  
 showed that full shifts 
  curries  full topological entropy and full Hausdorff dimension on the set of points with historic behavior.

A fundamental 
relation between historic behavior and  pointwise emergence
 is as follows.
 Let $\mathcal A_x\equiv \mathcal A_x(f)$ be  the set of accumulation points of $\{ \delta _x^n(f) \}_{n\geq 1}$ with respect to $ d $.
Notice that $x$ has historic behavior if and only if  $\# \mathcal A_x(f) >1$. 
For a subset $\mathcal Y$ of a compact metric space $\mathcal X$, 
let  $N (\epsilon , \mathcal Y)$ be  the $\epsilon$-covering number of $\mathcal Y$ by closed balls, 
 and  denote
the upper  and lower box-counting dimension of $\mathcal Y$  by $\overline{\mathrm{dim} }(\mathcal Y ) $  and $\underline{\mathrm{dim} }(\mathcal Y ) $ respectively,  that is, 
\[
\overline{\mathrm{dim} }(\mathcal Y  )= \limsup_{\epsilon \to 0} \frac{\log N(\epsilon , \mathcal Y  )}{-\log \epsilon} ,\quad \underline{\mathrm{dim} }(\mathcal Y  )= \liminf _{\epsilon \to 0} \frac{\log N(\epsilon , \mathcal Y  )}{-\log \epsilon} .
\]
When they coincide, we simply write it as $\mathrm{dim} (\mathcal Y  )$.
Then, 
it is straightforward to observe that
\begin{align}\label{eq:coveringpointwiseemergence}
\mathscr E_x(\epsilon ,f) & = N(\epsilon , \mathcal A_x(f)),
\end{align} 
and thus, 
by the well-known fact that $\mathcal A_x(f)$ is a  connected set in $\mathcal P(X)$ 
(cf.~\cite[Proposition 3.8]{DGS1976}),  
we conclude that $x$ has historic behavior  if and only if
 \begin{equation}\label{eq:prop1.2}
\liminf _{\epsilon \to 0} \frac{\log \mathscr E_x(\epsilon ,f )}{-\log \epsilon } =\underline{\mathrm{dim}}(\mathcal A_x(f)) \geq 1  
 \end{equation}
(in particular, $\mathscr E_x(\epsilon ,f )$ diverges as $\epsilon \to 0$).

From this dimensional perspective, 
we can easily get the following useful criterion for super-polynomial pointwise emergence (recall that each ergodic measure in $\mathcal P_f(X)$ is an extremal point; cf.~\cite[Proposition 5.6]{DGS1976}):
If 
 $\mathcal A_x(f)$ includes 
  the convex hull of  infinitely many distinct ergodic invariant probability measures, 
  then 
 \begin{equation}\label{eq:prop1.2b}
  \limsup _{\epsilon \to 0} \frac{\log \mathscr E_x(\epsilon ,f )}{-\log \epsilon }= \overline{\mathrm{dim}}(\mathcal A_x(f)) \geq\underline{\mathrm{dim}}(\mathcal A_x(f)) =\infty ,
 \end{equation}
that is, the pointwise emergence at $x$ diverges super-exponentially fast.
In the previously mentioned paper  by Sigmund \cite{Sigmund1974}, he in fact showed that if $f: X\to X$ satisfies the specification property, then  there exists a residual subset $R$ of $X$ such that  $\mathcal A_x(f)= \mathcal P_f(X)$ for any $x\in R$, 
where $\mathcal P_f( X)$ is the set of $f$-invariant probability measures on $X$. 
Furthermore, for any continuous map with the specification property, 
 the set of  periodic points is dense in $X$ (see \cite[Propositions 2]{Sigmund1974}), in particular, if $X$ is  an infinite set, then $\mathcal P_f(X)$ includes infinitely many distinct ergodic invariant probability measures (note that any dense subset of an infinite  metric space is an infinite set).
By combining  these results  with \eqref{eq:prop1.2b},
 we immediately get the following conclusion for super-polynomial pointwise emergence.   
\begin{prop}\label{thm:expanding}
Let $X$ be an infinite compact metric space  and $f: X\to X$ a continuous map with the specification property.
Then,   there is  a residual subset $R$ of $ X$ such that 
\begin{equation}\label{eq:0501a}
 \lim _{\epsilon \to 0}  \frac{\log \mathscr E_x(\epsilon ,f) }{-\log \epsilon } =\infty
  \quad \text{for all 
   $x\in R$}.
\end{equation}
\end{prop}

\begin{rmk}
 It seems that  super-polynomial pointwise emergence in \eqref{eq:0501a} holds for more general classes of  dynamical systems  without the specification property,  
 such as dynamics with historic behavior in  \cite{BLV2014,KLS2016,AP2018,BKNSR}. 
Moreover, it is of great interest to see whether one can develop \emph{thermodynamic
formalism
 on the set of points with super-polynomial pointwise emergence}, refer to e.g.~\cite{PP84,BS2000,Thompson2012}.  
\end{rmk}

\subsection{Metric and topological emergences}\label{ss:mte}
A  formula similar to \eqref{eq:coveringpointwiseemergence}  is seen in \cite[Proposition 3.14]{BB} for metric emergence: if $f$ is a conservative map, then
\begin{equation}\label{eq:0401}
\mathscr E_{\mathrm{Leb}}(\epsilon , f)  \leq N( \epsilon ,  \mathcal P_{\mathrm{erg}} (f) ),
\end{equation}
where  $\mathcal P_{\mathrm{erg}} (f)$  is the set of ergodic probability measures of $f$. 
In \cite{BB}, the quantity of the right-hand side of \eqref{eq:0401} is called the \emph{topological emergence} of $f$ at scale $\epsilon >0$, and its complexity and connection with metric emergence were deeply investigated. 
Conformal expanding repellers and hyperbolic sets of conservative surface diffeomorphisms  are  important examples for which the inequality in \eqref{eq:0401} is strict (\cite[Theorem A]{BB}).  
  Note that there is no inclusion relationship  between $\mathcal A_x$ and  $ \mathcal P_{\mathrm{erg}} (f)$ in general (see examples below), and so is between pointwise and topological emergences.   
A basic property 
 for 
  pointwise and metric emergences is the following (its proof will be given in Section \ref{s:pre}). 
\begin{lem}\label{lem:0402}
Let  $f: M\to M$ be a continuous map on a compact manifold $M$.
For any $\epsilon >0$ and Borel set $D\subset M$ of positive Lebesgue measure,  
\[
\min _{x\in D}\mathscr E_x(\epsilon ,f) \leq \mathscr E_{\mathrm{Leb}} (\mathrm{Leb} (D) \epsilon ,f ) .
\]
\end{lem}
We  summarize the differences between metric/topological and pointwise emergences: 
unlike the similarity in definition, the properties of these 
 emergences are rather in  strong contrast. 
Firstly, recall that if $f$ has infinitely many sinks, then its metric emergence $\mathscr E_{\mathrm{Leb}}(\epsilon )$ diverges in the limit $\epsilon \to 0$ (\cite[Claim 1.13]{Berger2017}). 
On the other hand, it is obvious that  for such dynamics $f$, the pointwise emergence is trivial (i.e.~$\mathscr E_x(\epsilon )=1$ for any $\epsilon >0$) on the basin of the sinks. 
Similarly, in \cite[Section 4]{BB}, many   conservative systems $f$ on the annulus 
 with super-polynomially diverging metric emergence were constructed, while, 
 since the constructed dynamics decomposes the annulus into $f$-invariant circles, 
one can easily see that pointwise emergence of the conservative system is minimal
  everywhere. 
   That is, the inequality in Lemma \ref{lem:0402} is strict for such dynamical systems.  
Conversely, by virtue of  Proposition \ref{thm:expanding}, 
 for expanding maps or Anosov diffeomorphisms on a compact manifold $M$, 
the pointwise emergence diverges super-polynomially fast
   on a residual subset of $M$, 
 while the metric emergence is bounded because there exist finitely many SRB   measures whose basins cover $M$  up to a zero Lebesgue measure set (cf.~\cite{BDV04}).

Furthermore,   Bowen's example has only finitely many ergodic probability measures under an appropriate 
setting (cf.~\cite{GGS2020}), so that its topological emergence is  bounded, 
while the pointwise emergence diverges 
with polynomial order of degree at least $1$
on a positive Lebesgue measure set
 because it has historic behavior and \eqref{eq:prop1.2} holds.
Conversely, for expanding maps or Anosov diffeomorphisms, 
the pointwise emergence is bounded almost everywhere, 
while the topological emergence may diverge super-polynomially fast due to \cite[Theorem A]{BB}.

\subsection{Main result}\label{s:mainresults}
By examples in Subsection \ref{subsection:hb} together with   \eqref{eq:prop1.2},  one can find  many dynamical systems whose  pointwise emergence diverges on a topologically or measure-theoretically large set. 
In fact, we saw in Proposition \ref{thm:expanding}
 that there are abundant dynamical systems with high pointwise emergence on a residual set, due to the established theory for $\mathcal A_x$. 
 However, 
in measure-theoretic context, 
 to the best of our knowledge,  
the box-counting dimension of $\mathcal A_x$
 (i.e.~the degree of polynomial  growth of pointwise emergence)  for any known dynamics with historic behavior is only $1$, 
 except  
  some   quadratic maps investigated by  Keller and Hofbauer 
     \cite{HK1990}. 
Furthermore, as mentioned in Subsection \ref{ss:mte}, any known dynamical system with super-polynomial metric emergence  is not  helpful to construct high pointwise emergence. 
However, we can show that there are abundant dynamical systems with super-polynomial pointwise emergences on a positive Lebesgue measure set, which is our main result.

Let $\mathrm{Diff} ^r(M)$ be the space of $\mathcal C^r$ diffeomorphisms on a closed surface $M$. 
In this paper, we mean by a basic set a compact hyperbolic and locally
maximal invariant set which is transitive and contains a dense subset of periodic
orbits. 
We tacitly assume throughout this paper that any basic set is not a single orbit.  
Recall that a non-empty connected open set $D$ is called a \emph{wandering domain} of $f$ if $f^i(D) \cap f^j(D) =\emptyset$ for all nonnegative integers $i, j$ with $i\neq j$. 
\begingroup
\setcounter{tmp}{\value{thm}}
\setcounter{thm}{0}
\renewcommand\thethm{\Alph{thm}}
\begin{thm}\label{thm:main}
There exists a dense subset $\mathcal D$ of any Newhouse open set (definition given in Section \ref{s:keydfn}) of $\mathrm{Diff} ^r(M)$  with 
 $2\leq r<\infty$   such that for each $f \in \mathcal D$, one can find a 
 wandering domain (in particular, a positive Lebesgue measure set)  
$D\subset M$ 
such that
the union of $\omega$-limit set of each point in $D$ includes a basic set $\Lambda$ 
 and 
\begin{equation}\label{eq:1110}
\mathcal A_x \supset \{ 
(1-\zeta ) \delta _{\hat p} + \zeta \mu \mid \mu \in \mathcal P_f(\Lambda )  \}
 \quad \text{for all $x\in D$} 
 \end{equation}
with some $\zeta\in (0,1]$ and   a saddle fixed point    $\hat p \in \Lambda$ (refer to Theorem \ref{thm:0411b}), where
 $\mathcal P_f(\Lambda ) $ is the set of $f$-invariant probability measures whose supports are included in $\Lambda$. 
Furthermore,
\begin{equation}\label{eq:0501b}
 \lim _{\epsilon \to 0}  \frac{\log \mathscr E_x(\epsilon ,f) }{-\log \epsilon } =\infty
  \quad \text{for all 
   $x\in D$} .
\end{equation}
\end{thm}
\endgroup

\setcounter{thm}{3}
Notice that 
\eqref{eq:0501b} is an immediate consequence of \eqref{eq:1110} due to \eqref{eq:prop1.2b}. 
By Lemma \ref{lem:0402}, 
we can contribute to the previously mentioned Berger program for metric emergence \cite[Problem 1.14]{Berger2017} as follows:
\begin{cor}\label{cor:1}
There exists a dense subset $\mathcal D$ of any Newhouse open set of $\mathrm{Diff} ^r(M)$ with a closed surface $M$ and $2\leq r<\infty$ such that for each $f \in \mathcal D$, 
\[
 \limsup _{\epsilon \to 0}  \frac{\log \mathscr E_{\mathrm{Leb}}(\epsilon ,f) }{-\log \epsilon } =\infty .
\]
\end{cor}

We emphasize that  Theorem \ref{thm:main} would be substantially  stronger than  Corollary \ref{cor:1},  because metric emergence quantifies infinitude of averages while  pointwise emergence quantifies non-existence of averages (historic behavior)
   as explained in Section \ref{ss:mte}. 
Therefore, 
we would rather say that 
Theorem \ref{thm:main} is a result about 
 a \emph{quantitative version of 
 Takens' Last Problem} 
 in the spirit of  
   Berger program.

\begin{rmk}
  Berger and Bochi proved in    \cite[Theorem D]{BB} that 
there exists a  residual  subset of any Newhouse open set 
of $\mathrm{Diff}^r(M)$ with $\mathrm{dim} (M)  =2$ and $1\leq r\leq \infty$ 
  such that any element of the  subset enjoys super-polynomial \emph{metric} emergence (although the result is proven by reducing it to  a conservative surface diffeomorphism whose  metric emergence is super-polynomial but    \emph{pointwise} emergence is everywhere minimal).
  That is, Corollary \ref{cor:1} is an alternative proof of a part of their result, and  so, it is natural to ask whether Theorem \ref{thm:main} holds with a residual set instead of the dense set $\mathcal D$. 
\end{rmk}

\begin{rmk}
After we completed the proof of super-polynomial pointwise emergence in Theorem \ref{thm:main}, 
we learned from Pierre Berger and S\'ebastien Biebler that they independently 
obtained a similar result, 
   although their proof is 
  quite  different  from ours 
(for example, they use a geometric model for parameter families of surface real mappings while we do not).
We also refer  to \cite{Berger2019, Talebi2020} for other recent results about emergences.
\end{rmk}

\subsection{Stretched exponential emergences}
\label{subsection:0120d}

In this subsection, we give a supplementary result related with    stretched exponential growths of pointwise emergences:
  we separated the results in this subsection from Section \ref{s:mainresults} because the proofs essentially use the results of Berger-Bochi \cite{BB} and  Berger-Biebler \cite{BB2020}, while the proof of Theorem \ref{thm:main} is self-contained.

Let $f: X\to X$ be a continuous map on a metric space $X$. 
  Let $\mathcal P_f( X)$ be the set of $f$-invariant probability measures on $X$ equipped with the first Wasserstein metric $d$.   
Then, it is not difficult 
 to see that 
$
 \mathcal A_x\subset \mathcal P_f(X)$ for all $x\in X$ 
 (cf.~\cite[Proposition 3.8]{DGS1976}).
 On the other hand, 
 it follows from \cite[Theorem 1.3]{BB}  that
 \begin{equation}\label{eq:0617a}
 \limsup _{\epsilon \to 0} \frac{\log \log N(\epsilon , \mathcal P_f (X) )}{-\log \epsilon } \leq  \overline{\mathrm{dim }}(X)
 \end{equation}
 (see Subsection 1.2 for the definition of $\overline{\mathrm{dim}}(X)$).
 So, it follows from \eqref{eq:coveringpointwiseemergence} 
  that 
\begin{equation}\label{eq:0615a}
\limsup _{\epsilon \to 0} \frac{ \log \log \mathscr E_x(\epsilon , f)}{-\log \epsilon } \leq  \overline{\mathrm{dim }}(X) \quad  \text{for any $x\in X$} .
 \end{equation} 
  Hence, it is natural to ask when the above inequality is an equality (i.e.~when stretched exponential pointwise emergences with maximal exponent are observed). 
 We refer to  \cite{BB} for a systematic study of stretched exponential growths for  topological and metric emergences.

We first give an answer for a special case of dynamics with the specification property, that is, subshifts  with the specification property, as follows.
\setcounter{thm}{4}
\begin{prop}\label{thm:0120b}
Let $X \subset   \{ 1, 2, \ldots , m\} ^{\mathbb N}$ be a subshift  $(m\geq 2)$, endowed with a standard metric $d_{X}(x ,y ) = \sum _{j=0}^\infty \frac{\vert x_j -y_j \vert }{\beta ^j}$ for $x=(x_0 ,x_1, \ldots ), y=(y_0, y_1, \ldots ) \in X$ with some $\beta >1$.
 Let $f:X \to X$ be the left shift operator.  
 Assume that $f$ satisfies the specification property.
 Then, there is a residual subset $R$ of $X$ such that
\[
\lim _{\epsilon \to 0} \frac{ \log \log \mathscr E_x(\epsilon , f)}{-\log \epsilon } = \mathrm{dim }(X) \quad  \text{for any $x\in R$} .
\]
\end{prop}
We recall   Furstenberg's formula $\mathrm{dim} (X) =  h_{\mathrm{top}}(f) / \log \beta $ for subshifts (\cite{Furstenberg1967}), where $h_{\mathrm{top}}(f) $ is the topological entropy of $f$. 
Using this formula, we will  show  that
\begin{equation}\label{eq:0604b}
\liminf _{\epsilon \to 0} \frac{ \log \log N(\epsilon , \mathcal P_f(X ))}{-\log \epsilon } \geq \dim (X)
\end{equation}
for $f$ in Proposition \ref{thm:0120b}, 
under the help of  Berger-Bochi's key estimate, see  Appendix \ref{s:is}.
Hence, Proposition \ref{thm:0120b} immediately follows from  \eqref{eq:coveringpointwiseemergence}, \eqref{eq:0615a} and  the aforementioned Sigmund's theorem for continuous maps with the specification property. 
  Note that \eqref{eq:0617a} and \eqref{eq:0604b} also imply the stretched exponential \emph{topological} emergence  with the maximal exponent $\dim (X)$ because $\mathcal P_{\mathrm{erg}}(f)$ is dense in $\mathcal P_f(X)$ for any continuous map on a compact metric space with the specification property (cf.~\cite{DGS1976}).
Furthermore,  one may get a two-sided subshift version of Proposition \ref{thm:0120b} with a small effort, and (the version of) Proposition \ref{thm:0120b}  may be   directly applicable to   basic sets of conservative surface diffeomorphisms or  conformal repellers, 
 for which maximal topological emergence is known  in \cite[Theorem A]{BB}.

Next we  consider dynamics in  Newhouse open sets. 
We have no idea of whether $\liminf _{\epsilon \to 0} ( \log \log N(\epsilon , \mathcal P_f(\Lambda )))/(-\log \epsilon ) \geq \dim (\Lambda )$ holds
(compare with \cite[Theorem 2.4]{BB} in which $f$ is a \emph{conservative} surface diffeomorphism having a basic set).
On the other hand, we  have the following  result by Berger and Biebler.
\begin{thm}$\mathrm{(}$\cite[Theorem 0.4]{BB2020}$\mathrm{)}$
\label{thm:0120f}
Let $f: M\to M$ be a $\mathcal C^{r}$ surface diffeomorphism $(1< r\leq \infty )$ on a closed surface $M$ having a basic set $\Lambda$ of saddle type.\footnote{Here we mean by saddle type that the tangent bundle  $TM$ over $\Lambda$ can be decomposed into one-dimensional stable and unstable bundles. 
 Theorem 0.4 of \cite{BB2020} only dealt with the case when $\Lambda$ is a horseshoe. 
However, the proof was given by projecting the saddle-type basic set of the surface  to a repeller of the real line and applying \cite[Theorem A]{BB}, which
 is proven for any conformal repeller.
Thus, the proof of \cite[Theorem 0.4]{BB2020} can be applied literally to obtain Theorem \ref{thm:0120f}.} 
Then, 
\[
\liminf _{\epsilon \to 0} \frac{ \log \log N(\epsilon , \mathcal P_f(\Lambda ))}{-\log \epsilon } \geq   \mathrm{dim } _u (\Lambda ) .
\]
where $\mathrm{dim } _u (\Lambda )>0$ is the box-counting dimension of $\Lambda \cap W^u_{\mathrm{loc}} (x)$ for any $x \in \Lambda$.\footnote{
Note that any saddle-type basic set of  a surface 
 is  $u$-conformal  in the sense of \cite[Section 22]{Pesin2008}, so
it follows from \cite[Theorem 22.1]{Pesin2008} that 
 $\dim (W^u_{\mathrm{loc}}(x)\cap \Lambda)$ is independent  of $x\in \Lambda$ and positive.}
\end{thm}
Under the establishment 
 of
  Theorem \ref{thm:0120f},    Theorem \ref{thm:main} together with \eqref{eq:coveringpointwiseemergence}  implies
stretched exponential pointwise emergences with a \emph{positive}  exponent:
\begin{cor}\label{cor:0611}
Let $f: M\to M$ be the surface diffeomorphism
and $D\subset M$ the  wandering domain given in Theorem \ref{thm:main}.
Then,
\[
\liminf _{\epsilon \to 0} \frac{ \log \log \mathscr E_x(\epsilon , f)}{-\log \epsilon } \geq  \mathrm{dim }_u(\Lambda )  \quad  \text{for any $x\in D$} .
\]
\end{cor}

\section{Key definitions and outline of proof}\label{s:keydfn}

In this section, we provide key definition used in the proof of main theorem, and briefly explain outline of the proof.

\subsection{Preliminary definitions}\label{subsection:01202}
We first  give 
precise definitions to 
undefined
 terminologies in Section \ref{s:intro}. 
Let $X$ be a compact metric space endowed with a metric $d_X$. 
A continuous map $f: X\to X$ on $X$ is said to satisfy the \emph{specification property} if for any $\epsilon > 0$, there exists a constant $\tau (\epsilon )\geq 0$ such that for any integer $K\geq 1$, any points $p_1, \ldots , p_K \in X$, any  integers $n_1, \ldots , n_K$ and 
$m_1, \ldots , m_K$ satisfying 
  $m_k -n_k \geq \tau (\epsilon )$ for any $1\leq  k \leq K$, there exists a periodic point $x \in X$ of period $N_K$   such that, with $N_k =\sum _{j=1}^{k} m_j$ and $N_0=0$, for every $1 \leq  k \leq K$,
\[
f^{n}(x) \in B_\epsilon ( f^n(p_k))   \quad \text{if $n\in I_k:=[N_{k-1}, N_{k-1} +n_k-1]$}
\]
where  $B_r (y)$ is  the ball with radius $r >0$ and  centered at $y\in X$.
Refer to \cite{DGS1976}.

 For $j=1,2$, let $p_j :X\times X\to X $ be the canonical projection to the $j$-th coordinate, and $(p_j) _*\pi$ the pushforward measure of a probability measure $\pi$ on $X\times X$ by $p_j$. 
 Let $\Pi (\mu, \nu )$ be the set of probability measures $\pi$ on $X\times X$ such that $(p_1)_*\pi =\mu$ and $(p_2)_*\pi =\nu$. (Such a measure $\pi$ is called a \emph{transport plan} or \emph{coupling} from $\mu$ to $\nu$.)  
 The first Wasserstein metric $ W_1 $ is defined as
 \[
  W_1  (\mu ,\nu )= \inf _{\pi \in \Pi (\mu , \nu )} \int _{X\times X} d_X(x,y) d\pi (x,y) \quad \text{for $ \mu  ,\nu \in \mathcal P(X)$} .
 \]
 (The integral in this formula is called the \emph{cost} of the transport plan $\pi$ with
respect to the \emph{cost function} $d_X$.) 
A standard reference for  Wasserstein metric is \cite{Villani2003,Villani2008}.
What we need in this paper is the following Kantorovich-Rubinstein dual representation  of the first Wasserstein metric: 
\begin{equation}\label{eq:0404d2}
 W_1  (\mu ,\nu ) = \sup _{\varphi \in \mathrm{Lip} ^1(X, \mathbb R )} \left\vert \int _X \varphi (x) d\mu (x) - \int _X \varphi (x) d\nu (x) \right\vert ,
\end{equation}
   where     $\mathrm{Lip} ^1(X, \mathbb R)$ is the space of all Lipschitz continuous real-valued functions $\varphi$ on $X$ whose  Lipschitz constant of $\varphi$ is bounded by $1$.  
Recall that we denoted $W_1$ by $d$ in Section \ref{s:intro}.

Next let us define a Newhouse open set. 
Let $M$ be a closed surface. 
 It was shown by Newhouse  that, for any $g \in \mathrm{Diff}^r(M)$ ($r\geq 2$) with a homoclinic tangency of
a dissipative saddle fixed point $\hat p$, 
there is an open set $\mathcal O \subset  \mathrm{Diff} ^r(M)$ whose closure contains $g$ and such that any element 
 of $\mathcal O$ is arbitrarily $\mathcal C^r$-approximated by a diffeomorphism $f$ with a homoclinic tangency associated with a dissipative saddle fixed point $\hat p_f$ which is the continuation of $\hat p$, and moreover $f$ has a $\mathcal C^r$-persistent tangency associated with some basic sets $\Lambda _f$ containing $p_f$ (i.e.~there is a $\mathcal C^r$ neighborhood of $f$  any element of which has a homoclinic tangency for the continuation of $\Lambda _f$). 
Such an open set $\mathcal O$ is called a \emph{Newhouse open set} (associated with $g$),   and call $(\hat p_f, \Lambda _f)$ the \emph{associated pair} of $f$.  
See  \cite{Newhouse79}.

\subsection{ Infinite dimensional simplex } 
As seen in Subsection \ref{subsection:hb},  in order to show super-polynomial pointwise emergence, it suffices to prove that $\mathcal A_x$ includes  an infinite dimensional simplex. 
We here prepare some notation to explore  the idea in detail.  
For each sequence  $\mathcal J = \{ \mu ^{(\ell )} \}_{\ell \geq 0}$ of 
probability measures     on $X$,   we define $\Delta  (\mathcal J )$ 
    by
\[ 
 \Delta (\mathcal J ) = \bigcup _{L\geq 1} \Delta _L(\mathcal J ), \quad \Delta _L(\mathcal J )=  \left \{  \mu _{\mathbf t}  (\mathcal J ) \mid \mathbf t\in A_L  \right\} , 
\]
where 
\begin{equation}\label{eq:1008c}
A_L=\left\{ (t_0, t_1, \ldots ,t_L) \in [0,1] ^{L+1}  \mid  \sum _{\ell =0}^L t_\ell =1 \right\} 
\end{equation}
endowed with the Euclidean norm induced from $[0,1] ^{L+1}$, 
and
\begin{equation}\label{eq:1008b}
\mu _{\mathbf t} \equiv \mu _{\mathbf t} (\mathcal J ) = \sum _{\ell =0} ^L t_\ell  \mu ^{(\ell )} \quad \text{for $\mathbf t =(t_0, t_1, \ldots ,t_L)$}  \in A_L.
\end{equation}
We also define $E (\mathcal J ,f)  $ by
\[
E (\mathcal J ,f) = \left \{ x\in X \mid \Delta (\mathcal J ) \subset \mathcal A_x(f)
\right\} .
\]

\subsection{Homogeneous coding on a wandering domain}

 In the next  subsection, we will construct  just \emph{one} nice code such that the associated point is in $E (\mathcal J, f)$ with a  sequence $\mathcal J = \{ \mu ^{(\ell )} \} _{\ell \geq 0}$ of infinitely many ergodic invariant probability measures. 
So we need to ``enlarge'' the point to a positive Lebesgue measure set.
  Let $f: M\to M$ be a $\mathcal C^r$ diffeomorphism ($r\geq 1$)  on a compact manifold 
   $M$.
 The following is the key definition in the proof of Theorem \ref{thm:main},
which is    reminiscent of the  specification property

\begin{dfn}\label{def_Gamma_wandering}
Let $\hat p$ be a fixed point and   $\{ p^{(\ell )}\}_{\ell \geq 0}$ a sequence of periodic points. 
  Let  
 $\{ m_k\} _{k\geq 1}$ be 
  a sequence of positive integers and $\{ \ell _k\} _{k\geq 1}$ a sequence of nonnegative integers. 
We say that a wandering domain 
 $D\subset M$ of $f$ is 
  \emph{coded by  $\{ \ell _k  \} _{k\geq 1}$ for $( \hat p, \{ p^{(\ell )} \} _{\ell \geq 0})$ over the base order $\{ m_k\} _{k\geq 1}$}
if 
there exist
sequences  $\{\widehat I_k\}_{k\geq 1}$  and $\{I_k\}_{k\geq 1}$ of disjoint discrete intervals 
and a sequence   $\{\epsilon_k\}_{k\geq 1}$ of positive numbers with $\lim_{k\rightarrow\infty}\epsilon_k=0$  
 satisfying the following conditions.
\begin{enumerate}[(1)]
\item[$\mathrm{(C1)}$] 
For any  $k\in \mathbb N$, 
 \[
\widehat I_k \cup I_k \subset [N_{k-1} , N_k -1]  ,
\]
where $N_k=\sum_{j=1}^{k}  m_j$ and $N_0=0$. 
Furthermore,  
\[
\lim_{k\to \infty}\frac{\# \widehat I_k+ \# I_k}{ m_k} =1 . 
\]
\item[$\mathrm{(C2)}$] \label{D_3}
For any sufficiently large $k\in \mathbb{N}$,  
\[
 f^n(D)\subset B_{\epsilon_k}\left(\hat p \right)\; \text{ if $n\in \widehat I_k$}. 
 \]
\item[$\mathrm{(C3)}$]  
 For any sufficiently large $k\in \mathbb{N}$, 
 \[
   f^n(D)\subset B_{\epsilon_k}\left( f^n (p^{(\ell _k)} )\right ) \; \text{ if $n\in I_k$},
 \] 
and 
 $\# I_k$ is a multiple of  $ \mathrm{per} (p^{(\ell _k)}) $.  Furthermore, 
 \[
\zeta =  \lim_{k\to \infty} \frac{\# I_k }{m_k} \quad \text{exists as a  strictly positive number}. 
\]
 \end{enumerate}
See Figure \ref{f_line2}.
\end{dfn}

\begin{figure}[hbt]
\centering
\scalebox{0.6}{\includegraphics[clip]{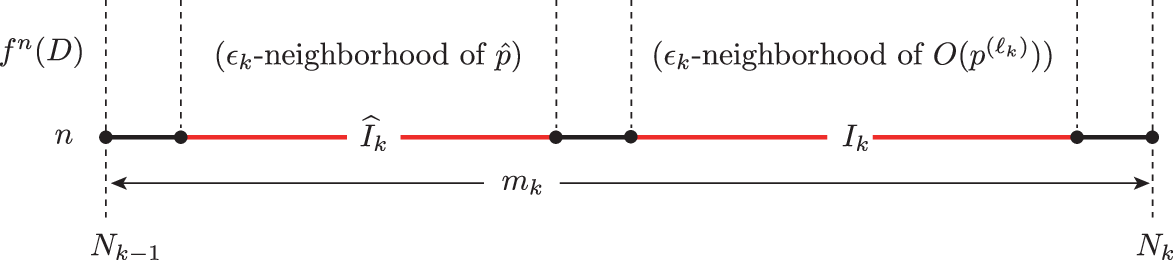}}
\caption{Travel of $f^n(D)$ for $n\in [N_{k-1},N_k)$.}
\label{f_line2}
\end{figure}

We say that a sequence of positive integers $\{ m_k\}_{k\geq 1}$ is \emph{moderate} if 
\begin{equation}\label{eq:moderate}
\lim _{k\to \infty} \frac{m_k}{N_k} =0.
\end{equation}
The following theorem is  
a key generalization  of 
 the idea behind a number of  estimates in 
\cite{KS2017} to the context of pointwise emergences
and will be used in the proof of Theorem \ref{thm:main}.

\begin{thm}\label{thm:2}
Let $M$ be a closed surface and  $2\leq r<\infty$.
For any Newhouse open set $\mathcal O\subset \mathrm{Diff}^r(M)$, 
 any element $\tilde f\in \mathcal{O}$ having an associated pair $(\hat p_{\tilde f}, \Lambda _{\tilde f})$,
  any neighborhood
$\mathcal{U}(\tilde f)$  of $\tilde f$ in $\mathcal O$, 
   and any     sequence of periodic points  $\{ p^{(\ell)} _{\tilde f}\}_{\ell \geq 0} \subset \Lambda _{\tilde f}$,  
there exists a
moderate sequence of positive integers  $\{ m_k \} _{k\geq 1}$ such that  
for any sequence of nonnegative
integers $\{\ell _k\}_{k\geq 1}$, one can find an element $f$ of $\mathcal{U}(\tilde f)$ 
having 
  a wandering domain 
  coded by  $\{ \ell _k \} _{k\geq 1}$ for $(\hat p_f, \{ p^{(\ell )} _f\} _{\ell \geq 0})$ over $\{ m_k\} _{k\geq 1}$,
where 
 $\hat p_ f$ and  $p^{(\ell)}_f$ are the continuations  of $p_{\tilde f}$ and $p^{(\ell )}_{\tilde f}$ for   $\ell \geq 0$, respectively. 
\end{thm}

\subsection{Adapted code}\label{subsection:findacode}

The final 
 step of the proof of Theorem \ref{thm:main} is to 
construct an adapted code in the following sense.
Denote by $\mathrm{per}(p)$ the period of a periodic point $p$.

\begin{thm}\label{thm:0411b}
  Let $f: M\to M$ be a $\mathcal C^r$ diffeomorphism on a compact 
    manifold $M$ with $r\geq 1$. 
For 
  any moderate sequence of positive integers $\{ m_k\} _{k\geq 1}$, any fixed point $\hat p$ and any sequence of periodic points $\{p^{(\ell)} \}_{\ell \geq 0}$, there is a  sequence  of nonnegative integers $\{ \ell _k\} _{k\geq 1} $ 
  such that the following holds: 
  Suppose that $f$ has a wandering domain $D$  coded by $\{ \ell _k  \} _{k\geq 1}$ for $(\hat p, \{p^{(\ell)}\}_{\ell \geq 0})$ over $\{ m_k\} _{k\geq 1}$.
Then we have 
\[
D \subset E (\{  \mu ^{(\ell)}\} _{\ell \geq 0} ,f) \quad \text{ with }
     \quad
     \mu ^{(\ell)} = (1-\zeta ) \delta _{\hat p}  + \zeta   \delta _{p^{(\ell )}}^{\mathrm{per}(p^{(\ell )})}, \; \ell \geq 0,
     \]
where $\zeta$  is the positive number given  in Definition \ref{def_Gamma_wandering}. 
\end{thm}

\subsection{Plan of the proof}

Observe that 
\[
N( \epsilon , \Delta (\{ (1-\zeta )\hat \mu + \zeta \mu ^{(\ell)} \}_{\ell \geq 0} )) = N(\zeta ^{-1}\epsilon , \Delta ( \{ \mu ^{(\ell)} \}_{\ell \geq 0} ) )
\]
 for any $\epsilon >0$, $\zeta \in (0,1]$, $\hat \mu \in \mathcal P(X)$   and $\{ \mu ^{(\ell )}\} _{\ell \geq 0} \subset  \mathcal P(X)$.
Hence,   \eqref{eq:1110}  immediately follows from  Theorems 
 \ref{thm:2} and  \ref{thm:0411b} by taking $ \{p^{(\ell)}_f\}_{\ell \geq 0})$ as
 \[
  \mathrm{Per} (f\vert _{\Lambda _f}) =\{ \hat p_f\} \cup \Big(\bigcup _{\ell \geq 0} O(p^{(\ell)} _f) \Big),
  \]
  where   $O(p^{(\ell )} _f)$ is the forward orbit of $p^{(\ell )}_f$ and   $\mathrm{Per} (f \vert _{\Lambda _f})$  is 
the set of periodic points of $f$ on  $\Lambda _f$.
Furthermore, as mentioned, \eqref{eq:0501b} is a consequence of \eqref{eq:prop1.2b} and \eqref{eq:1110}. 
Therefore, we get the conclusion of Theorem \ref{thm:main}.
We give the proof of Theorem \ref{thm:2} in Section \ref{section:4} and the proof of Theorem \ref{thm:0411b} in Section \ref{section:3}.

\section{Preliminary}\label{s:pre}
In this section we prove Lemma \ref{lem:0402}, together with some basic properties of  $\delta _x^n$ with respect to $ d $ that will be used in the following sections.
Let $f: X\to X$ be a continuous map on a compact metric space $X$ equipped with a metric $d_X$. 
Recall that $d$ is the first Wasserstein metric on $\mathcal P(X)$. 
  Since $X$ is compact, we can assume that $d_X(x,y) \leq 1$ for all $x, y\in X$ without loss of generality, and
the Kantorovich-Rubinstein
dual representation \eqref{eq:0404d2} implies that for each probability measures $\mu, \nu$ on $X$, 
\begin{equation}\label{eq:0404d}
 d(\mu ,\nu ) = \sup _{\varphi \in \mathrm{Lip} ^1(X, [0,1] )} \left\vert \int _X \varphi (x) d\mu (x) - \int _X \varphi (x) d\nu (x) \right\vert ,
\end{equation}
 where     $\mathrm{Lip} ^1(X, [0,1] )$ is the space of functions $\varphi$ on $X$  with values in $[0,1]$ such that the Lipschitz constant of $\varphi$ is bounded by $1$ (notice that for any $\varphi\in\mathrm{Lip} ^1(X, \mathbb R)$, one can find a constant $a$ such that $\varphi +a \in\mathrm{Lip} ^1(X, [0,1])$).  
\begin{lem}\label{lem:reset}
For any $m> n\geq 1$ and $x\in X$, we have
\[
 d \Big(\delta _x^m , \delta _{f^n(x)} ^{m-n} 
\Big) \leq \frac{2n}{m} .
\]
\end{lem}
\begin{proof}
For any continuous function $\varphi :X\to \mathbb R$ with value in $[0,1]$, 
\begin{align*}
\left\vert \int _X\varphi d\delta _x ^{m} - \int _X\varphi  d\delta _{f^n(x)} ^{m-n}\right\vert
\leq \left\vert \left( \frac{1}{m} - \frac{1}{m-n}\right) \sum _{j=n}^{m-1} \varphi (f^j(x)) \right\vert
+\left\vert  \frac{1}{m} \sum _{j=0}^{n-1} \varphi (f^j(x)) \right\vert ,
\end{align*}
which is bounded by $\frac{2n}{m}$, so we get the conclusion due to \eqref{eq:0404d}.
\end{proof}

The next lemma follows from a similar argument. 
\begin{lem}\label{lem:reset5}
For any $n\geq 1$ and $x\in X$, we have
\[
 d\big(\delta _x^n , \delta _{x} ^{n+1} 
\big) \leq \frac{2}{n+1} .
\]
\end{lem}

We also recall the following basic fact, 
refer to e.g.~\cite[Section 7.2]{Villani2003}. 
\begin{lem}\label{lem:reset2}
For each $x, y\in X$, 
\[
 d (\delta _x ,\delta _y ) = d_X(x, y) .
\]
\end{lem}

Finally, we will use the following lemma.
\begin{lem}\label{lem:reset2b}
For each $L\geq 1$ and $ \mathbf t, \mathbf s\in A_L$, we have  
\[
 d (\mu _{\mathbf t} ,\mu _{\mathbf s} ) \leq (L+1) \vert \mathbf  t - \mathbf s\vert ,
\]
where $\mu _{\mathbf t}$ and $A_L$ are given in \eqref{eq:1008c} and \eqref{eq:1008b}.
\end{lem}
\begin{proof}
For any $\varphi \in \mathrm{Lip} ^1(X, [0, 1])$ and $\mathbf t=(t_0,\ldots ,t_L)$, $\mathbf s=(s_0,\ldots ,s_L)$ in $A_L$ with $L\geq 1$,
\begin{align*}
\left\vert \int _X \varphi d\mu _{\mathbf t} -  \int _X \varphi d\mu _{\mathbf s} \right\vert &\leq \sum _{\ell =0} ^L \left\vert  t_\ell -s_\ell \right\vert  \int_X\left\vert \varphi \right\vert d\mu ^{(\ell)}   \\
&\leq (L+1)\max _{0\leq \ell \leq L} \vert t_\ell -s_\ell \vert \leq (L+1)\vert \mathbf t -\mathbf  s\vert ,
\end{align*}
which implies the conclusion due to \eqref{eq:0404d}. 
\end{proof}

\begin{proof}[Proof of Lemma \ref{lem:0402}]
Fix $\epsilon >0$ and a  positive Lebesgue measure set $D$. 
Let $N_0$ be the maximal integer $N$ such that 
for all probability measures  $\{ \mu _j \}_{j=1}^N$ on $M$, the inequality in \eqref{eq:0403d} does not hold for any $x\in D$. Note that $\min _{x\in D}\mathscr E_x(\epsilon ,f) = N_0+1$.

Given probability measures $\{ \mu _j \}_{j=1}^{N_0}$ on $M$, by definition of $N_0$, we get
\[
\limsup _{n\to \infty}  \min _{1\leq j\leq N_0}d (\delta _x ^ n, \mu _j ) >  \epsilon  \quad \text{for any $x\in D$} .
\]
Therefore, it follows from Lebesgue's dominated convergence theorem together with \cite[Proposition 1.10]{Berger2017} that 
\begin{align*}
\limsup _{n\to \infty} \int _M  \min _{1\leq j\leq N_0}d(\delta _x ^ n, \mu _j ) d\mathrm{Leb} (x) 
> \mathrm{Leb}  (D) \epsilon . 
\end{align*}
Thus, the inequality in \eqref{dfn:metricemergence} with $\mathrm{Leb}(D) \epsilon $ instead of $\epsilon$ is not  satisfied by the probability measures $\{ \mu _j \}_{j=1}^{N_0}$, 
implying that 
$\mathscr E_{\mathrm{Leb}}( \mathrm{Leb}(D) \epsilon ) \geq N_0 +1$. 
This completes the proof. 
\end{proof}

\section{Proof of Theorem \ref{thm:2}}\label{section:4}

\subsection{Preliminary}\label{s:0312z}

The proof of Theorem \ref{thm:2} is based on the argument  in our previous work \cite[Theorem A]{KS2017}: 
the most important part is a modification of \emph{Critical Chain Lemma}  (\cite[Lemma 7.1]{KS2017}), 
but  the modified version of Critical Chain Lemma   can be  
 proven as the original version. 
In the rest of this section, we first briefly recall necessary definitions and dynamics in \cite{KS2017}, 
and precisely  describe how we should modify   Critical Chain Lemma, 
together with a short explanation for the reason why  the modification does not affect  the proof of Critical Chain Lemma.
Finally, we will complete the proof of Theorem \ref{thm:2}, by translating  the  argument 
in the proof of  Theorem A of  \cite{KS2017} after the establishment of  Critical Chain Lemma
  into our context.

\subsection*{Notation}
In this section, we will use the notation $N_0 , N_1,N_2$,  which are positive integers borrowed from \cite{KS2017}, but different from $N_k$ defined in Definition \eqref{def_Gamma_wandering}. 
To avoid notational confusion, we use $\widetilde N_k$ for an integer playing the role of $N_k$ in Definition \eqref{def_Gamma_wandering} (refer to \eqref{eqn_mkNk}).

\vspace{0.2cm}

Let $M$ be a closed surface and  $2\leq r<\infty$.
Let $\tilde f$ be an element of a Newhouse open set $\mathcal O\subset \mathrm{Diff}^r(M)$. 
By definition of Newhouse open sets,  $\tilde f$ has
 a dissipative saddle fixed point $\hat p _{\tilde f}$
 and a basic set $\Lambda _{\tilde f}$ such that $\hat p _{\tilde f}\in \Lambda _{\tilde f}$ and 
 $\tilde f$ has a persistent homoclinic tangency associated with $\Lambda _{\tilde f}$. 
In fact, $\tilde f^k$ has a basic set $\widetilde{\Lambda }$ on which $f^k$ is conjugate to a two-sided full shift of two symbols  $\{ 1, 2\}^{\mathbb Z}$  and $\Lambda = \bigcup _{j=0}^{k-1} \tilde f^j(\widetilde{\Lambda } )$ with some $k\in \mathbb N$. 
For simplicity, we assume that $f\vert _\Lambda$ is conjugate to the two-sided full shift of two symbols. 
We also fix a small neighborhood $\mathcal U(\tilde f)$ of $\tilde f$ in $\mathcal O$. 
Then, 
one can find  an element $f$ of $\mathcal{U}(\tilde f)$ which has the continuations $\hat p_f$ of $\hat p _{\tilde f}$ and $\Lambda _f$ of $\Lambda _{\tilde f}$ such that 
\begin{itemize}
\item[(S-i)] $ \Lambda _f$ contains $\hat p_f$; 
\item[(S-ii)] $f$ has
a quadratic tangency $q_f$ associated with $\hat p_f$; 
\item[(S-iii)] $f$ is linear in $U(\hat p_f) \cap f^{-1}(U (\hat p_f))$ with a small neighborhood $U(\hat p_f)$ of $\hat p_f$.
\end{itemize}
We refer to e.g.~\cite{PT1995} (compare  also with Section 3 in \cite{KS2017}).

We suppress $f$ from 
the  notations 
   $\hat p_f$, $\Lambda _f$ and $q_f$. 
By replacing the basic set $\Lambda$  by a smaller one if necessary, we can choose the linearizing coordinate in (S-iii) such that 
$\Lambda\subset S\subset S ^\prime $
 where $S=[0,2]\times [0,2]$ and $S ^\prime =[-2,2]\times [-2,2]$ with $\hat p=(0,0)$. 
Set 
$W_{\mathrm{loc}}^s(\hat p)=[-2,2]\times \{0\}$ and $W_{\mathrm{loc}}^u(\hat p)=\{0\}\times [-2,2]$.
Let $\mathcal{F}_{\mathrm{loc}}^{s}(\Lambda)$ and $\mathcal{F}_{\mathrm{loc}}^{u}(\Lambda)$ be a local stable foliation and a local unstable foliation on $S$ compatible 
with $W_{\mathrm{loc}}^{s}(\Lambda )$ and $W_{\mathrm{loc}}^{u}(\Lambda )$, respectively.
For $\sigma = s, u$, consider the projection $\pi^\sigma : S \to W_{\mathrm{loc}}^\sigma (\hat p)$ along the leaves of $\mathcal{F}_{\mathrm{loc}}^{\overline \sigma} (\Lambda)$ where $\overline s=u$ and $\overline u =s$, and  
 the Cantor set $K_\Lambda^{\sigma }=\pi^{\sigma }(\Lambda) $ in $W_{\mathrm{loc}}^{\sigma }(\hat p)$.

Let  $\sigma = s$ or $u$. 
We denote by $B^{\sigma }(0)$ the smallest interval in $W_{\mathrm{loc}}^{\sigma }(\hat p)$ containing $K^{\sigma }_\Lambda$.
There exists a Markov partition of $B^{\sigma }(0)$ for $K_\Lambda^{\sigma}$ which consists of sub-intervals $B^{\sigma}(1;1)$, $B^{\sigma }(1;2)$ 
of $B^{\sigma }(0)$ 
with $\partial B^{\sigma }(0)\cap B^{\sigma }(1;1)=\{ \hat p\}$.
Let $\Psi^{\sigma }:B^{\sigma }(1;1)\sqcup B^{\sigma }(1;2)\to B^{\sigma }(0)$ be the map defined 
by $\Psi^s=\pi^s\circ f^{-1}$ and $\Psi^u=\pi^u\circ f$, which is $\mathcal C^{1+\alpha}$ for some $0<\alpha<1$. 
For each integer $\ell \geq 1$ and $w_i\in \{1,2\}$ for $1\leq i\leq \ell$, 
we define the interval $B^{\sigma }(\ell ;w_1\cdots w_\ell)$, called a $\sigma $-\emph{bridge} of the $\ell$-th \emph{generation}, by
$$B^{\sigma }(\ell ; w_1\cdots w_\ell )=\bigl\{x\in B^{\sigma }(0)\, \mid \,(\Psi^{\sigma })^{i-1}(x)\in B^{\sigma }(1;w_i),\, i=1,\dots , \ell
\bigr\}.$$
Here we say that the word $(w_1\cdots w_\ell )$ is the \emph{itinerary} for the $\sigma$-bridge.
From the definition, we have 
\begin{equation}\label{eqn_Bsu}
\Psi^{\sigma }(B^{\sigma }( \ell ;w_1w_2\cdots w_\ell ))=B^{\sigma }(\ell -1;w_2\cdots w_\ell ).
\end{equation}

Let $B^{\sigma}$ be a $\sigma$-bridge  with $\sigma =s$ or $u$. 
The closure of a connected component of $B^\sigma \setminus K^\sigma _\Lambda$ is called a \emph{gap} of $K^\sigma _\Lambda$ in $B^\sigma$. 
Finally, we call $\mathbb{B}^{\sigma }=(\pi^{\sigma })^{-1}(B^{\sigma })$ 
 the 
\emph{bridge strip} of  $B^{\sigma}$, and 
$\mathbb{G}^{\sigma }=(\pi^{\sigma })^{-1}(G^{\sigma })$ 
 the \emph{gap strip} of $G^{\sigma}$ 
 (see Subsection 4.2 of \cite{KS2017} for details).

 \subsection{The dynamics in \cite{KS2017}}\label{s:0312a}
We used in \cite{KS2017} two preliminary perturbations (Section 3, 4, 5), and two main perturbations (Section 5, 7).
 One of the preliminary perturbations 
  is given for the perturbed dynamics to satisfy
 the conditions 
  (S-i), (S-ii), (S-iii). 
 The other 
  is given to satisfy   (S-iv), (S-v), (S-vi), (S-vii) of \cite[Section 3]{KS2017}. 
  We now 
  let   $f$ be  an element of $\mathcal U(\tilde f)$ satisfying all the conditions from (S-i) to (S-vii) (note that we here use the notation $f$  for $f_{\mu _n}$  in (S-iv), (S-v), (S-vi), (S-vii) 
 of \cite{KS2017} with large integer $n=n_*$ given in \cite[\S 5.2]{KS2017}). 
We merely remember that (S-iv), (S-v) and (S-vi) lead to the existence of another basic set $\Gamma $  of $f$ near the homoclinic  tangency $q$ such that 
\begin{itemize}
\item $\Lambda$ and $\Gamma $ are homoclinically related: both $W^u(\Lambda ) \cap W^s(\Gamma )$ and $W^s(\Lambda ) \cap W^u(\Gamma )$ contain non-trivial transverse intersections, 
\item there exists
 a heteroclinic tangency curve $L$ 
 between 
 $\Lambda$   and 
  $\Gamma$: 
  there are a smooth arc $L$, 
   a local stable foliation $\mathcal F_{\mathrm{loc}}^s (\Gamma )$ of $\Gamma$ which is compatible  with a local stable manifold of $\Gamma$ on a compact region $E$ containing $\Gamma$, and positive integers $N_0$, $N_2$ such that 
 $L \cap f^{-N_0} (\mathcal F_{\mathrm{loc}}^s(\Gamma ) ) 
  = L$ and 
$L \cap f^{N_2} (\mathcal F_{\mathrm{loc}}^u(\Lambda ))$ is a sub-arc of $L$ each element of which is a quadratic tangency of 
 $f^{N_2} (\mathcal F_{\mathrm{loc}}^u(\Lambda ))$ and   $f^{-N_0} (\mathcal F_{\mathrm{loc}}^s(\Gamma ))$. 
   \end{itemize} 
   Refer to \cite[\S 5.1]{KS2017} and see Figure \ref{f_bstrip}. 
(We note that the basic set $\Gamma$ is written as $\Gamma _m$ in \cite{KS2017}, where $m$ is the period of a periodic point included in $\Gamma _m$.) 
 
 The  key dynamics in \cite{KS2017}  is 
  the  return map on ($L$ and)  a neighborhood $U(L)$ of    $L$, 
  which is the composition  of 3 dynamics: 
 \begin{itemize}
 \item[(a)] the transient dynamics  from $U(L)$ to $\Lambda$ through 
 the homoclinic relation between $\Lambda$ and $\Gamma$, 
 \item[(b)] the hyperbolic dynamics on $\Lambda$, 
 \item[(c)] the transient dynamics from $\Lambda$ to $U(L)$ through 
the unstable foliation  $f^{N_2} (\mathcal F_{\mathrm{loc}}^u(\Lambda ))$. 
 \end{itemize}
To be precise, we borrow more notations from \cite{KS2017}. 
Let  $z_0$ be the positive integer satisfying (8.5) in \cite[Subsection 8.1]{KS2017} and $\{ z_k\} _{k\geq 1}$ arbitrary sequence of integers such that $z_k \in \{ z_0, z_0+1\}$ for each $k\geq 1$. 
Let  $\widehat{\underline{w}}_{\,k} $ be the itinerary given in \cite[Lemma 7.1]{KS2017} for  $k\geq 1$. 
The itinerary originates from Linear Growth Lemma (\cite[Lemma 6.1]{KS2017}), which implies  that if we denote  the length of $\widehat{\underline{w}}_{\,k} $ by $\widehat n_k$, then
\begin{itemize}
\item there is a constant $\alpha >0$ such that the $\alpha \vert B^s (\widehat n_k ; \widehat{\underline{w}}_{\,k} )\vert$-neighborhood of $ B^s (\widehat n_k ; \widehat{\underline{w}}_{\,k} )$ and $\alpha \vert  B^s (\widehat n_{k+1} ; \widehat{\underline{w}}_{\,k +1} )\vert $-neighborhood of $B^s (\widehat n_{k+1} ; \widehat{\underline{w}}_{\,k+1} )$ are disjoint,
\item  $\widehat n_k$  is of order $k$ (so the lemma is called Linear Growth Lemma). 
\end{itemize}
Moreover,
for arbitrary sequence $\{ \underline{v}_{\,k}\} _{k\geq 1}$ of itineraries with $ \underline{v}_{\,k+1} \in \{ 1, 2\} ^{k}$, 
we consider the bridge $B_{k}^{u} \equiv B_k^u(z_k ,\underline{v}_{\,k+1} )$ of $K^u_\Lambda$ and the bridge $B_{k+1}^{s*} \equiv B_{k+1}^{s*}(z_k , \underline{v}_{\,k+1})$
of $K^s_\Lambda$ given by
\begin{equation}\label{eqn_CC1}
\begin{split}
B_{k}^{u}&=B^u(z_kk^2+k^2+k+ \widehat n_{k+1};\,\underline{1}^{(z_kk^2)}\underline{2}^{(k^2)}\,\underline{v}_{\,k+1}\,[\widehat{\underline{w}}_{\,k+1}]^{-1}),\\
B_{k+1}^{s*}&=B^s(z_kk^2+k^2+k+\widehat n_{k+1};\,\widehat{\underline{w}}_{\,k+1}\,[\underline{v}_{\,k+1}]^{-1}\underline{2}^{(k^2)}\,\underline{1}^{(z_kk^2)}),
\end{split}
\end{equation}
where $\underline{1}^{(\ell )} $ (resp.~$\underline{2}^{(\ell )}$) is the itinerary consisting  of only $1$ (resp.~$2$) with length $\ell $ and $[\underline{w}]^{-1} = (w_{\ell} \cdots w_2 w_1)$ for each $\underline{w} = (w_1 w_2\cdots w_\ell) \in \{ 1,2 \} ^\ell$. 
It follows from Subsection 5.2 of \cite{KS2017} (refer also to Lemma 7.1~(2) of \cite{KS2017}) that there are a connected component $S _k \subset S$,   integers $N_1$ and $\widehat i_k$ of order $k$ for each $k\geq 1$ such that 
 $f^{-(\widehat i_k + N_1)}\left(S _k\cap ([0,1] \times \{t \} )\right)$ is 
  the intersection of $E$ and a leaf of $\mathcal F_{\mathrm{loc}} ^s(\Gamma)$ for each $t\in [0,1]$. (In terms and notations of \cite{KS2017}, $S _k$ is the sub-strip  of $S$ such that $f^{-(\widehat i_k+N_1)}(S _k) $ is the bridge stripe of a $u$-bridge $\widehat A^u_k$ of $\Gamma$ along $\mathcal F_{\mathrm{loc}} ^s(\Gamma)$, and denoted by $S(\widehat A^u_k)$.) 
We let the arc $L$ transversely return to itself by backward iterations of $f$ as
\begin{multline*}
 \widetilde L = f^{-N_2} (L \cap f^{N_2}(\mathcal F^u_{\mathrm{loc}}(\Lambda )) ), \\
 \widetilde L_k = f^{- (z_kk^2+k^2+k+\widehat n_{k+1})
 }(\mathbb B^{s*}_{k+1} \cap \widetilde L), \quad 
 L_k =f^{- (N_0 + \hat i_k + N_1)} (\widetilde L_k \cap S _k) ,
 \end{multline*}
so that  $ L$ and  $L_k$ have a transverse intersection $x_k$ for any large $k$ (cf.~\cite[\S 7.2]{KS2017}). 
Then, noting that $f^{z_kk^2+k^2+k+\widehat n_{k+1}} (\mathbb B_{k}^{u} ) = \mathbb B_{k+1}^{s*}$ by construction, we get 
\begin{itemize}
\item[(a)] $\widetilde x_k = f^{N_0 + \hat i_k + N_1}(x_k) \in \mathbb B_{k}^{u}$, mapped from $L$ to $\Lambda$, 
\item[(b)] $\widehat x_k =f^{z_kk^2+k^2+k+\widehat n_{k+1}}(\widetilde x_k )\in  \mathbb B_{k+1}^{s*}$, mapped on $\Lambda$,
\item[(c)] $f^{N_2}(\widehat x_k) \in L$, mapped from $\Lambda$ to $L$.
\end{itemize}
The second main perturbation 
(i.e.~the  perturbation in Subsection 7.2 of \cite{KS2017}) are made, with the notation $f$ again for the perturbed dynamics, 
 to get the relation
\begin{equation}\label{eq:1109}
f^{m_k}(x_k) = x_{k+1}, \quad \text{$m_k = N_2 + ( z_kk^2+k^2+k+\widehat n_{k+1}) + ( N_0 + \hat i_k + N_1 )$}
\end{equation}
for all   large $k$,
 and Critical Chain Lemma ensures that the perturbation can be  arbitrary small
 (see also the next subsection).
 
Finally, let $R_k$ be the  rectangle  given in Subsection 8.2 of \cite{KS2017}, where the center of $R_k$ is  $x_k$ and $R_k \cap R_{k^\prime} =\emptyset$ for each $k\neq k^\prime$. 
We notice that the distance between $x_k$ and $x_{k+1}$ is  large in the sense  of the first item of the above  properties of  $\widehat{\underline{w}}_{\,k} $ (and similar property for the $u$-bridges of $\Gamma$  in \cite[Lemma 6.1]{KS2017}).
Indeed,
Rectangle Lemma (\cite[Lemma 8.2]{KS2017}) states that $f^{m_k}(R_k) \subset R_{k+1}$
for sufficiently large $k$, and thus $D=R_k$ with a large $k$ is  a wandering domain. 

Furthermore, 
both $N_0 + \hat i_k + N_1$ and $N_2$ as well as the length of $\widehat{\underline{w}}_{\,k+1}\,[\underline{v}_{\,k+1}]^{-1}$ 
are at most of order $k$, while the lengths of $\underline{1}^{(z_kk^2)}$ and $\underline{2}^{(k^2)}$ are of order $k^2$ 
(so we called the dynamics (a) and (c) \emph{transient}). 
Therefore,
we can  find sequences of discrete intervals $\{ \widehat I_k\} _{k\geq 1}$  and 
 $\{ I_k\} _{k\geq 1}$, and  
 a sequence of positive measure $\{ \epsilon _k\} _{k\geq 1}$ with $\lim _{k\to \infty} \epsilon _k =0$ such that 
 \begin{equation}\label{eq:1109b}
 \widehat I_k \cup I_k \subset [\widetilde N_{k-1} , \widetilde N_k -1] \quad \text{where $\displaystyle \widetilde N_k=\sum _{j=1}^k m_j$} ,
 \end{equation}
 \begin{equation}\label{eq:1109f}
\lim _{k\to \infty}  \frac{\# \widehat I_k + \# I_k}{m_k} =1, \quad \lim _{k\to \infty} \left\vert  \frac{\# \widehat I_k}{m_k} - \frac{z_k }{z_k+1} \right\vert =0 ,
\end{equation}
and for any sufficiently large $k$,
 \begin{equation}\label{eq:1109g}
f^n (D) \subset B_{\epsilon _k} (\hat p) \quad \text{if $n\in \widehat I_k$} \quad \text{and } \quad f^n (D) \subset B_{\epsilon _k} (p^\prime ) \quad \text{if $n\in I_k$},  
\end{equation}
where $p^\prime$ is the another saddle fixed point of $f$ in $\Lambda$.
Refer to 
 Subsection \ref{subsection:B3} for the calculation, and
compare with Definition \ref{def_Gamma_wandering}.

 \subsection{Modification in Critical Chain Lemma}\label{s:0312b}

Let $\pi ^s : B^s(0) \to L$ be the projection along the leaves of $f^{N_2}(\mathcal F_{\mathrm{loc}} ^u (\Lambda))$ and $\pi _k^u : B^u(0) \to L$  the projection given by
\[
\pi _k ^u= \pi ^u \circ f^{-( \widehat i_k +N_1)} \circ \pi _{S_k} ,
\]
where $\pi ^u :E\to L$ is the projection  along the leaves of $f^{-N_0} (\mathcal F_{\mathrm{loc}}^s(\Gamma ))$ and $\pi _{S_k} $ is the projection from $B^u (0) $ to a component of the boundary of $S_k$ along the leaves of $f^{(N_1+\widehat i_k)}(\mathcal F^s_{\mathrm{loc}}(\Gamma))$. 
Let $B_{k,L}^u = \pi _k^u (B_k^u)$ and $B_{k,L}^{s*} =\pi ^s (B_k^{s*})$ (cf.~\cite[\S 5.1]{KS2017}).

Critical Chain Lemma states that there are constants $\epsilon _0 >0$, $r>1$ and an interval $J_k\subset (-\epsilon _0 r^{-k}, \epsilon _0 r^k)$ such that
\begin{equation}\label{eq:1109e2}
(B_{k+1,L}^{s*}  +t ) \cap B_{k,L}^u \neq \emptyset \quad \text{if and only if $t\in J_{k+1}$}.
\end{equation}
The second main perturbation is of the form 
\begin{equation}\label{eq:1109e}
\widehat x_k +u_k = f^{-N_2}(x_{k+1}) \quad \text{on $\widetilde L$}
\end{equation} 
to obtain \eqref{eq:1109}, and \eqref{eq:1109e2} 
ensures that $\vert u_k\vert $ is of order $r_1^{-k}$ with some $r_1>1$.

Let $\widehat{\underline{u}}_{\,k}$ be arbitrary sequence of itineraries with  $\widehat{\underline{u}}_{\,k} \in \{ 1, 2\} ^{k^2 +3k+1}$.
We modify Critical Chain Lemma by replacing $B_k^u$ and $B_{k+1}^{s*}$ in \eqref{eqn_CC1} with
$B_k^u \equiv B_k^u(z_0 , {\underline{u}}_{\,k})$ and $B_{k+1}^{s*} \equiv B_{k+1}^{s*}(z_0, {\underline{u}}_{\,k})$ 
given by 
\begin{equation}\label{eqn_CC0}
\begin{split}
B_{k}^u&=B^u(z_0k^2+k^2+3k+1+\widehat n_{k+1};\,\underline{1}^{(z_0k^2)}\,\widehat{\underline{u}}_{\,k}\,[\widehat{\underline{w}}_{\,k+1}]^{-1}),\\
B_{k+1}^{s*}&=B^s(z_0k^2+k^2+3k+1+\widehat n_{k+1};\,\widehat{\underline{w}}_{\,k+1}\,[\widehat{\underline{u}}_{\,k}]^{-1}\,\underline{1}^{(z_0k^2)})  .
\end{split}
\end{equation}
In \cite{KS2017}, 
the itinerary $\underline{2}^{(k^2)}$ together with the integer $z_k$ in  \eqref{eqn_CC1} is 
 chosen  as 
 the 
  wandering domain 
$D$ consists of points with historic behavior (recall \eqref{eq:1109b}, \eqref{eq:1109f} and  \eqref{eq:1109g}), 
and the itinerary $\underline{v}_{\,k+1}$ is used just to show that the $\omega$-limit set of the forward orbit 
contains $\Lambda$.
In Theorem \ref{thm:2}, 
such properties are not required.
So, all $z_k$ are unified to $z_0$, and $2^{(k^2)}$ and $\underline{v}_{\,k+1}$ are deleted.
However, for the proof of the existence of a wandering domain, it is crucial that 
the orbit stays long time in a small neighborhood of $\hat p$, and that the distance between $x_k$ and $x_{k+1}$ are sufficiently large.
So the roles of the itineraries $\underline{1}^{(z_0k^2)}$ and $\widehat{\underline{w}}_{\,k+1}$ are indispensable.
On the other hand, the itinerary $\underline{2}^{(k^2)}$ can be replaced by any itinerary of 
length $k^2+O(k)$.
Hence one can use any itinerary $\widehat{\underline{u}}_{\,k}$ with $|\,\widehat{\underline{u}}_{\,k}|=k^2+3k+1$ instead of $\underline{2}^{(k^2)}$.

\subsection{The end of the proof of Theorem \ref{thm:2}}\label{subsection:B3} 
Here we set 
$\widehat j_k =N_0 + \widehat i_k +N_1$ and 
\begin{equation}\label{eqn_mkNk}
m_k=N_2 + (z_0k^2 +k^2 +3k+1 +\widehat n_{k+1}) +\widehat j_k, 
\quad \widetilde N_k=\sum_{j=1}^k m_j
\end{equation}
(instead of \eqref{eq:1109} and \eqref{eq:1109b}).
As in the proof of Theorem A in \cite{KS2017}, there exists an element $f$ of $\mathcal{U} (\tilde f)$ which 
has a contracting wandering domain $D$ such that $f^{\widehat j_k+ \widetilde  N_{k-1}}(D)$ is contained 
in the gap strip $\mathbb{G}_{k}^u$ for all sufficiently large $k$. 
Since the second perturbation  is made only in the interior of 
$\mathbb{G}^u(0)$ (in fact $\widetilde L$ is included in $\mathbb{G}^u(0)$, recall also  \eqref{eq:1109e}), 
$\Lambda$ and $\hat p$ do not change by the perturbation.
See Figure \ref{f_bstrip}.
\begin{figure}[hbt]
\centering
\scalebox{0.6}{\includegraphics[clip]{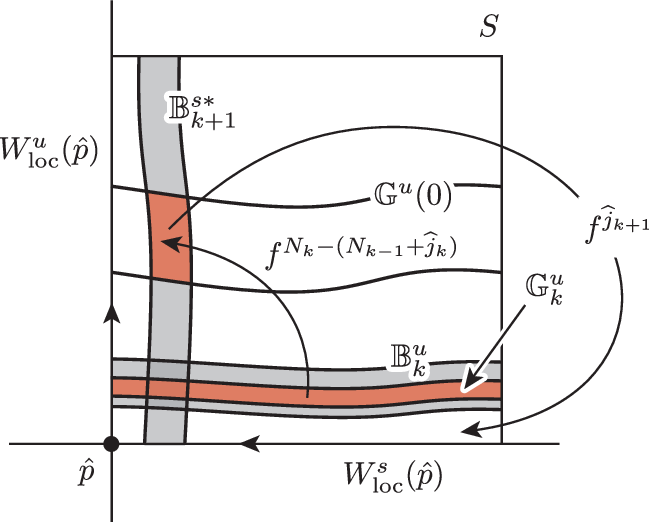}}
\caption{Travels of $D$ by $f$.}
\label{f_bstrip}
\end{figure}
According to Lemma 7.1 in \cite{KS2017}, $\widehat{\underline{w}}_{\,k+1}$ is 
the itinerary of length $\widehat n_{k+1}=O(k)$ which is arranged such that $f^{\widetilde N_k}(D)\subset \mathbb{B}_{k+1}^{s*}
\cap \mathbb{G}^u(0)$ 
is sent into $\mathbb{G}_{k+1}^u\subset \mathbb{B}_{k+1}^u$ by $f^{\widehat j_{k+1}}$.
Since $m_k=O(k^2)$ and $\widetilde  N_k=O(k^3)$, we have $\lim_{k\to \infty} m_k/ \widetilde N_k=0$.
This means that $\{m_k\}_{k\geq 1}$ is moderate.

Now we take a sequence $\{\ell_k\}_{k\geq 1}$ of non-negative integers arbitrarily.
For any non-negative integer $a$ with $\{k\,|\,\ell_k=a\}\neq\emptyset$, set  $\kappa(a)=\min\{k\,|\,\ell_k=a\}$.
Let $p^{(\ell_k)}$ be the periodic point of $\Lambda$ corresponding to the bi-infinite itinerary 
$(\underline{y}_{\,k})^{\mathbb{Z}}$, where
$$\underline{y}_{\,k}=1\cdot\underline{2}^{(\kappa(a))},$$
for $a=\ell _k$. 
Then
\begin{equation}\label{eqn_per_p}
\mathrm{per}(p^{(\ell_k)})=\kappa(a)+1\leq k+1.
\end{equation}
If $a\neq a'$, then $\{k\,|\,\ell_k=a\}\cap \{k'\,|\,\ell_{k'}=a'\}=\emptyset$ and 
hence $\kappa(a)\neq \kappa(a')$.
It follows that the orbit of $p^{(\ell_k)}$ is disjoint from that of $p^{(\ell_{k'})}$ if $\ell_k\neq\ell_{k'}$. 
By \eqref{eqn_per_p}, there exists an integer $s(k)$ with $k^2\leq s(k)\leq k^2+k+1$ which 
is a multiple of $\mathrm{per}(p^{(\ell_k)})$.
Then $q(k)=s(k)/\mathrm{per}(p^{(\ell_k)})$ is a positive integer.
Consider itineraries $\underline{a}_{\,k}$ and $\underline{b}_{\,k}$ with 
$|\,\underline{a}_{\,k}|=k$,  $|\,\underline{b}_{\,k}|=k^2+2k+1-s(k)$ and such that 
$\widehat{\underline{u}}_{\,k}=\underline{a}_{\,k}(\underline{y}_{\,k})^{(q(k))}\underline{b}_{\,k}$ 
is a sub-itinerary of $(\underline{y}_{\,k})^{(3q(k))}$.
From our definition, we have $k\leq |\,\underline{b}_{\,k}|\leq 2k+1$ and 
$|\,\widehat{\underline{u}}_{\,k}|=k^2+3k+1$, the latter of which is one of our required conditions.
Consider the discrete intervals $\widehat I_k=[\widetilde N_{k-1}+\widehat j_k+k+1,\widetilde N_k'-k]$ and 
$I_k=[\widetilde N_k'+k+1,\widetilde N_k-|\,\underline{b}_{\,k}|-\widehat n_{k+1}]$, where $\widetilde N_k'=z_0k^2 + \widehat j_k + \widetilde N_{k-1}$.
See Figure \ref{f_line}.
\begin{figure}[H]
\centering
\scalebox{0.56}{\includegraphics[clip]{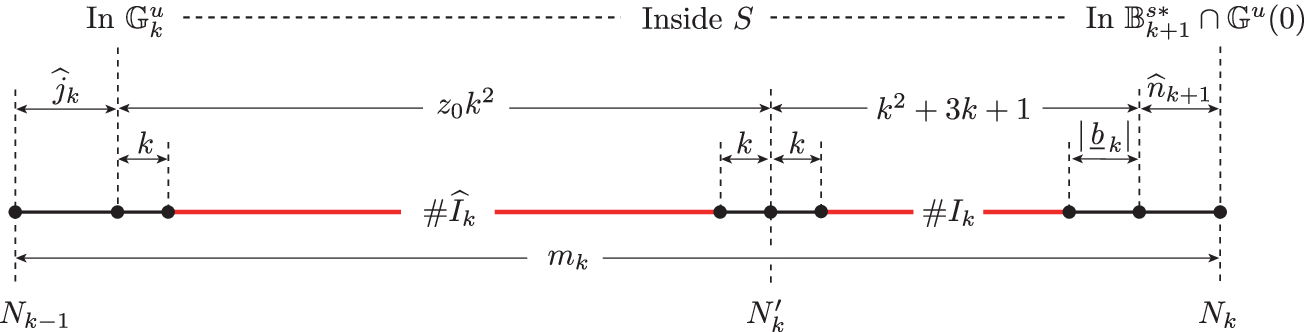}}
\caption{Locations of $f^j(D)$.}
\label{f_line}
\end{figure}
Then we have
$$\# I_k=s(k)=q(k)\mathrm{per}(p^{(\ell_k)}).$$
Moreover, since $m_k=(z_0+1)k^2+O(k)$ and $\# I_k=k^2+O(k)$, 
$$\zeta:=\lim_{k\to \infty}\frac{\# I_k}{m_k}=\frac1{z_0+1}>0.$$
By \eqref{eqn_Bsu}, for any $j\in \widehat I_k$, 
$f^j(D)\subset\mathbb{B}^u(k;\,\underline{1}^{(k)})\cap \mathbb{B}^s(k;\,\underline{1}^{(k)})$.
Similarly, for any $j\in I_k$, $f^j(D)\subset\mathbb{B}^u(k;\,y_{k,j})\cap \mathbb{B}^s(k;\,[y_{k,j}]^{-1})$ 
for some sub-itinerary $y_{k,j}$ of $(y_k)^{(q(k))}$ of length $k$.
Since the diameters of $\mathbb{B}^u(k;\,\underline{1}^{(k)})\cap \mathbb{B}^s(k;\,\underline{1}^{(k)})$ and 
$\mathbb{B}^u(k;\,y_{k,j})\cap \mathbb{B}^s(k;\,[y_{k,j}]^{-1})$ uniformly converge to zero as $k\rightarrow\infty$, 
$f$ satisfies the property \eqref{D_3} of Definition \ref{def_Gamma_wandering}.
This completes the proof of Theorem \ref{thm:2}.

\section{Proof of Theorem \ref{thm:0411b}}\label{section:3}

In this section,  we will prove Theorem \ref{thm:0411b}.
Let $f$ be a $\mathcal C^r$ diffeomorphism on a closed surface $M$. 
We fix a moderate sequence of positive integers $\{ m_k\}_{k\geq 1}$, 
a fixed point $\hat p$ and a sequence of 
periodic points $\{ p^{(\ell )} \}_{\ell \geq 0}$
 throughout the rest of this section.

\subsection{Reduction to $A_L$}\label{S:3.1} 
We start the proof of Theorem \ref{thm:0411b} 
by approximating the empirical measures $\{ \delta _x^n\} _{n\geq 0}$ along the orbit of $x$ in  a wandering domain with a code by 
measures $\mu _{\mathbf t}$ with parameters $\mathbf t\in A_L$ induced by 
   the code.

 \begin{dfn}\label{dfn:1008b}
For each  finite increasing sequence  $\mathbf k = \{ k(\ell )\} _{\ell =-1}^L$
of   positive integers  (i.e.~$k(\ell -1)<k(\ell )$ for $\ell \in [0,L]$),
we say that a sequence of nonnegative integers $\{  \ell _k\}_{k\geq 1}$ is \emph{associated  with $\mathbf k$}  
 if 
\[
\ell _k = 
\ell \quad \text{for all  $\ell \in [0, L]$  and $k\in [  k  ( \ell -1) +1 ,  k ( \ell ) ]$}.
\]
See Figure \ref{f1008} for the travel of $f^n (D)$ for $D$ coded by $\{ \ell _k\} _{k\geq 1} $ associated with $\{ k(\ell )\} _{\ell =-1}^L$ (compare with Definition \ref{def_Gamma_wandering} and Figure \ref{f_line2}).
\end{dfn}

\begin{figure}[hbt]
\centering
\scalebox{0.6}{\includegraphics[clip]{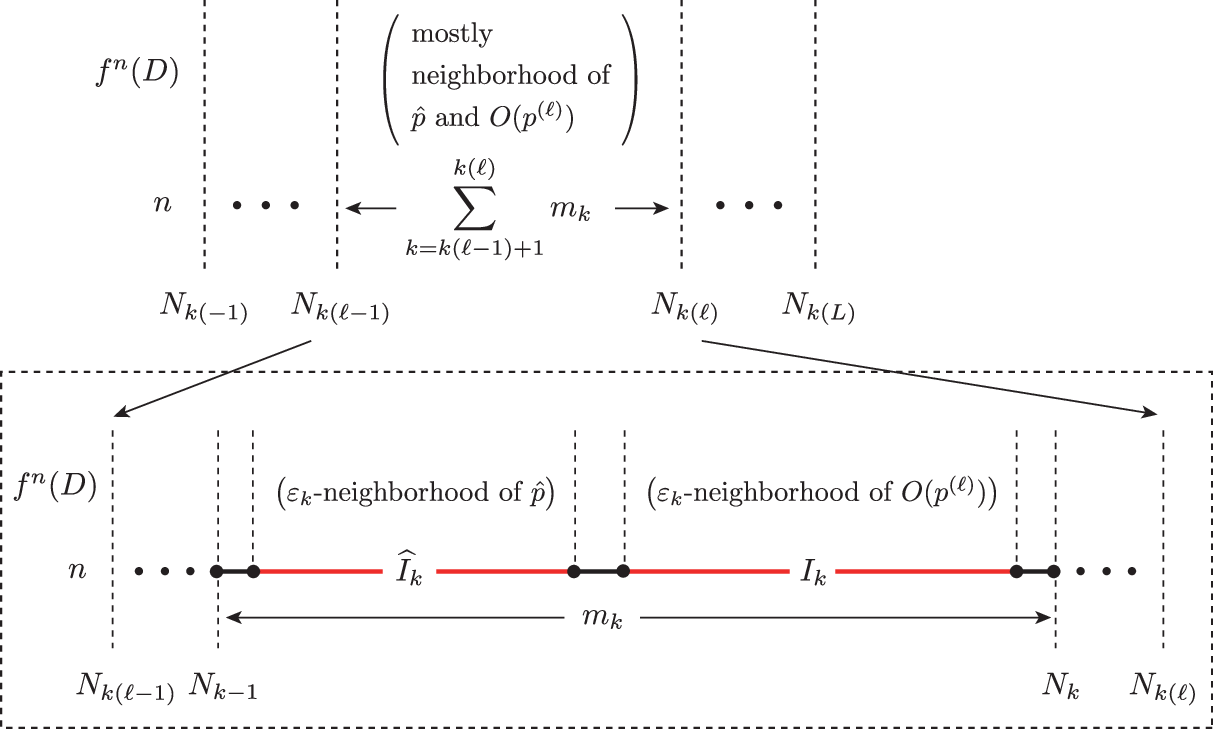}}
\caption{Travel of $f^n(D)$ for  $D$ coded by $\{ \ell _k\} _{k\geq 1}$ associated with $\{ k(\ell ) \} _{\ell =-1}^L$}
\label{f1008}
\end{figure}

For each  finite increasing sequence 
of  positive integers  $\mathbf k = \{ k(\ell )\} _{\ell =-1}^L$, we define $\overline{ M} (\mathbf k)\in
  \mathbb N^{L+1}$ by
\[
\left(\overline{ M}( \mathbf k) \right)_\ell = 
\sum _{k =k (\ell -1 )+1}^{k (\ell ) } m_k \quad \text{for $\ell \in [0,L]$.}
\]
Furthermore, for each $\mathbf{M} =(M_0, \ldots ,M_L)\in \mathbb N^{L+1}$, we define $\bar t (\mathbf M) \in  A_L$ by 
\begin{equation}\label{eq:0316b}
\bar t( \mathbf M ) = \left( \frac{M_0}{S_L} , \frac{M_1}{S_L} , \ldots , \frac{M_L}{S_L} \right) \quad \text{with} \quad S_\ell = M_0+\cdots +M_\ell .
\end{equation}
For a  wandering domain $D$ coded by a sequence of nonnegative integers over $\{ m_k\} _{k\geq 1}$, let $k_D$ be the minimal integer such that (C2) and (C3) in Definition \ref{def_Gamma_wandering} hold for all $k\geq k_D$.  
Recall the notation $N_k =\sum _{j=1}^{k} m_j$. 

\begin{lem}\label{prop:13}
Let $L\geq 1$ 
and  $\mathbf k = \{ k(\ell )\} _{\ell =-1}^L$ a finite increasing sequence 
of   positive integers. 
Let  $\{  \ell _k\}_{k\geq 1}$
 be a sequence of nonnegative integers associated with $\mathbf k$. 
 Let $f$ be  a $\mathcal C^r $ diffeomorphism  having  a wandering domain $D$  coded by $\{ \ell _k\}_{k\geq 1}$  for $( \hat p, \{ p^{(\ell )} \} _{\ell \geq 0})$ over $\{ m_k\} _{k\geq 1}$. 
 Suppose that $k(-1) \geq k_D$.
Then,
 for any $x\in D$, 
 \begin{multline*}
 d \left( \delta _x^{N_{k(L)}} (f) , \mu _{\overline t (\overline M(\mathbf k ))} (f ) \right) 
\leq 
 \frac{2N_{k(-1)}}{N_{k(L)}} + 2 \max_{k\in [k( -1) +1, k(L )]} \epsilon _{ k} \\
 +2\max _{0\leq \ell \leq L} \left\vert \frac{\sum _{k=k(\ell -1)+1}^{k(\ell )} \# I_k}{\sum_{k=k(\ell -1)+1}^{k(\ell )}m_k} -\zeta \right\vert +2\max _{0\leq \ell \leq L}\left\vert 1 - \frac{\sum _{k=k(\ell -1)+1}^{k(\ell )}( \#I_k + \# \widehat I_k)}{\sum _{k=k(\ell -1)+1}^{k(\ell )} m_k} \right\vert ,
\end{multline*}
where $\mu _{\mathbf t} = \sum _{\ell =0}^L t_\ell \mu ^{(\ell)}$ for $\mathbf t=(t_0,\ldots ,t_L)$ and $\mu ^{(\ell)}$ is given in Theorem \ref{thm:0411b}.
\end{lem}

\begin{remark}\label{rmk:tildeepsilon}
It is easy to see that by (C1) and (C3), both of
\[
\sup _{k_2>k_1+1} \left\vert\frac{\sum_{k=k_1+1}^{k_2}  \#I_k}{\sum _{k=k_1+1}^{k_2} m_k} -\zeta \right\vert  \quad 
\text{and}
\quad \sup _{k_2>k_1+1} \left\vert 1 - \frac{\sum_{k=k_1+1}^{k_2}  (\#I_k + \# \widehat I_k)}{\sum _{k=k_1+1}^{k_2} m_k}  \right\vert 
\]
go to $0$ as $k_1\to \infty$. 
\end{remark}

\begin{proof}[Proof of Lemma \ref{prop:13}]
Fix $x\in D$. 
For each $k\geq 1$, let 
\[
s_{1, k} = \sum _{n\in \widehat I_k} \delta _{f^n(x)}, \quad
 s_{2, k} = 
   \sum _{n\in I_k} \delta _{f^n(x)}, \quad 
 s_{3, k} = 
 \sum _{n\in [N_{k-1}, N_k-1] - \widehat I_k - I_k} \delta _{f^n(x)}.
\]
Let  $s_{j} (\ell ) =\sum _{k=k(\ell -1)+1}^{k(\ell ) } s_{j, k}
$  for each $j=1, 2, 3$ and $\ell \geq 0$.
Then, we have a decomposition 
\begin{equation}\label{eq:0314a1}
\sum _{j=N_{k(-1)}} ^{N_{k(L)} -1} \delta _{f^j(x)} 
= 
\sum _{\ell =0}^L 
(s_{1}(\ell )+s_{2} (\ell ) +s_{3} (\ell )) . 
\end{equation}
Note also that, for each $j=1,2, 3$, 
\begin{equation}\label{eq:0314a2}
\frac{\sum _{\ell =0}^L s_j(\ell )}{N_{k(L) } -N_{k(-1)}} = \sum _{\ell =0}^L \frac{s_j(\ell ) }{\sum _{k=k(\ell -1)+1}^{k(\ell ) } m_k} \times t_\ell , 
\end{equation}
where $t_\ell =\frac{\sum _{k=k(\ell -1)+1}^{k(\ell ) } m_k}{\sum _{\ell =0}^L\sum _{k=k(\ell -1)+1}^{k(\ell ) } m_k}$, and that $\overline t (\overline M(\mathbf k) )=(t_0, \ldots ,t_L)$.

Fix $\varphi \in \mathrm{Lip} ^1(M, [0,1])$. Since $f^n(x) \in B_{\epsilon _k}(\hat p)$ for  all $k\geq 1$ and $n\in \widehat I_k$, it follows from  Lemma \ref{lem:reset2} that 
\[
\left\vert \frac{ \int _X \varphi ds _1(\ell ) }{ \sum _{k=k(\ell -1)+1}^{k(\ell ) }  m_k } -  \frac{ \sum _{k=k(\ell -1)+1}^{k(\ell ) }  \# \widehat I_k}{ \sum _{k=k(\ell -1)+1}^{k(\ell ) }  m_k }  \int _X\varphi d\delta _{\hat p}  \right\vert
 \leq \max_{k(\ell -1)+1 \leq k\leq k(\ell)} \epsilon _k
\]
for every $\ell \in [0,L]$. 
So we have  
\begin{multline}\label{eq:0314a3}
 \left\vert \frac{\int _X \varphi ds_1 (\ell )}{ \sum _{k=k(\ell -1)+1}^{k(\ell ) } m _k} -  (1-\zeta )\int _X\varphi d\delta _{\hat p}  \right\vert \\
 \leq \max_{k(\ell -1)+1 \leq k\leq k(\ell)} \epsilon _k 
  +   \left\vert \frac{ \sum _{k=k(\ell -1)+1}^{k(\ell ) } \# I_k}{ \sum _{k=k(\ell -1)+1}^{k(\ell ) } m_k } -\zeta \right\vert
  + \left\vert  1-  \frac{  \sum _{k=k(\ell -1)+1}^{k(\ell ) } (\# \widehat I_k +\# I_k )}{ \sum _{k=k(\ell -1)+1}^{k(\ell ) } m_k} \right\vert . 
\end{multline}

Similarly,  
since $f^n(x) \in B_{\epsilon _k}(f^n( p^{(\ell _k)} ))$ for  all  $k\geq 1$ and $n\in I_k$,  
  and 
 $\# I_k$ is a multiple of 
$ \mathrm{per} (p^{(\ell _k)})$, we get that 
\[
\left\vert \int _X \varphi d   s_{2,k }  -  \# I_k \int _X \varphi d \mu ^{(\ell _k)}\right\vert 
\leq \#  I_k \epsilon _k \quad \text{for every $k\geq 1$}.
\]
Therefore, by the assumption that  
  $\{ \ell _k\} _{k\geq 1}$ is associated with $ \mathbf k$, 
\[
\left\vert \frac{\int _X \varphi ds_2 (\ell ) }{ \sum _{k=k(\ell -1)+1}^{k(\ell ) } m _k} -  \frac{\sum _{k=k(\ell -1)+1}^{k(\ell ) }  \#  I_k   }{ \sum _{k=k(\ell -1)+1}^{k(\ell ) } m _k} \int _X\varphi d\mu ^{(\ell)} \right\vert 
 \leq \max_{k(\ell -1)+1 \leq k\leq k(\ell)} \epsilon _k 
\]
for every $\ell \in [0,L]$, and 
 we have  
\begin{equation}\label{eq:0314a4}
\left\vert \frac{\int _X \varphi ds_2 (\ell ) }{ \sum _{k=k(\ell -1)+1}^{k(\ell ) } m _k} - \zeta \int _X \varphi d \mu ^{(\ell)} 
\right\vert 
  \leq \max_{k(\ell -1)+1 \leq k\leq k(\ell)} \epsilon _k 
  +   \left\vert \frac{ \sum _{k=k(\ell -1)+1}^{k(\ell ) } \# I_k}{ \sum _{k=k(\ell -1)+1}^{k(\ell ) } m_k } -\zeta \right\vert . 
\end{equation}
Furthermore, it is easy to check that 
\begin{equation}\label{eq:0314a5}
  \left\vert \frac{\int _X \varphi ds_3(\ell )}{ \sum _{k=k(\ell -1)+1}^{k(\ell ) } m _k} \right\vert \leq 
\left\vert 1-  \frac{  \sum _{k=k(\ell -1)+1}^{k(\ell ) } (\# \widehat I_k +\# I_k )}{ \sum _{k=k(\ell -1)+1}^{k(\ell ) } m_k} \right\vert . 
\end{equation}

By \eqref{eq:0314a1}, \eqref{eq:0314a2}, \eqref{eq:0314a3}, \eqref{eq:0314a4} and \eqref{eq:0314a5},  together with Lemma \ref{lem:reset}, we immediately get the conclusion.  
\end{proof}

\subsection{Filling of $A_L$}\label{s:def}

The following lemma is elementary but crucial. 
Notice that the choice of $\mathbf k$ 
  is \emph{independent of  both $f$ and $(\hat \mu, \{  \mu ^{(\ell )} \} _{\ell \geq 0})$}.

\begin{lem}\label{lem:12}
For any positive  integer $L$, nonnegative number $\tilde c$, positive number $\epsilon$ and  $\mathbf t \in A_L$, there is an increasing sequence of  positive integers $\mathbf k = \{ k(\ell )\} _{\ell =-1}^L$ with $k(0)>\tilde c$ such that 
\[
\left\vert \overline{ t} (\overline M (\mathbf k)) -\mathbf t\right\vert \leq \epsilon .
\]
\end{lem}

Lemma \ref{lem:12} easily follows from the following lemma. 
 For   $\mathbf M=(M_0, \ldots ,M_L)\in \mathbb N^{L+1}$, 
  we define $\overline{ T} (\mathbf M) \in [0,1]^L$ by
\begin{equation*}
\overline{ T}( \mathbf M) = \left( \frac{M_1}{S_1}, \frac{M_2}{S_2} , \ldots , \frac{M_L}{S_L} \right) .
\end{equation*}
(Recall \eqref{eq:0316b} for $S_\ell$.)

\begin{lem}\label{lem:11}
For any positive  integer $L$, nonnegative number $\tilde c$, positive number $\epsilon$ and  $\mathbf T \in [0,1]^L$, there is an increasing sequence of  positive integers $\mathbf k = \{ k(\ell )\} _{\ell =-1}^L$ with $k(0)>\tilde c$ such that 
\[
\left\vert \overline{ T} ( \overline M (\mathbf k) )-\mathbf T\right \vert \leq \epsilon .
\]
\end{lem}

\begin{proof}
We  use the notation $N_{k ', k }=\sum _{j=k'+1}^{k } m_j$ for $k'<k$, so that  
 we have
\[
\left( \overline T  ( \overline M(\mathbf k)) \right) _\ell = \frac{N_{k(\ell-1) , k(\ell)}}{N_{ k(-1), k(\ell )}} \quad \text{for $\ell \in [1,L]$}.
\]
Fix $L\geq 1$, $\tilde c\geq 0$, $\epsilon >0$ and $\mathbf T = (T_1, T_2, \ldots , T_L )\in [0,1]^L$. 
By the assumption \eqref{eq:moderate} for the moderate sequence $\{ m_k\} _{k\geq 1}$, we can take an integer $k(0)> \tilde c$ such that 
\begin{equation}\label{eq:0315e}
\frac{m_{k'}}{N_{k(-1), k'}} \leq \frac{ \epsilon }{\sqrt L} \quad \text{ for any $k'> k(0)$.}
\end{equation} 
Therefore, $ \frac{N_{k(0) , k(0)+1}}{N_{k(-1),k(0)+1} } =  \frac{m_{ k(0)+1}}{N_{k(-1),k(0)+1} } \leq \frac{\epsilon }{\sqrt L}$, and 
for each $k'>k(0)$, 
\begin{align*}
 \frac{N_{k(0) , k'}}{N_{k(-1),k'} } -  \frac{ N_{k(0), k'-1}}{N_{k(-1), k'-1}}
=\frac{N_{k(-1), k(0)} m_{k'}}{N_{k(-1), k'} N_{k(-1), k'-1}} 
\leq \frac{\epsilon }{\sqrt L} .
\end{align*}
Moreover, $\mathbb N \ni k^{''}\mapsto \frac{ N_{k(0) ,k(0)+k^{''}} }{ N_{k(-1), k(0)+k^{''}}}$ is monotonically increasing with value in $(0,1)$. 
So,  
 there is a positive integer $k(1)>k(0)$ such that 
$
\left\vert  \frac{N_{k(0),k(1)}}{N_{k(-1), k(1)}}  - T_1 \right\vert \leq \frac{\epsilon }{\sqrt  L} .
$

Assume that one can find  $k(\ell )>
 \ldots >k(0)$ satisfying  $\left\vert  \frac{N_{k(j-1),k(j)}}{N_{k(-1), k(j)}}  - T_j \right\vert \leq \frac{\epsilon }{\sqrt  L}$ for every $j\in [1,\ell ]$. 
Then, by virtue of \eqref{eq:0315e},
$ \frac{N_{k(\ell) , k(\ell )+1}}{N_{k(-1),k(\ell )+1} }   \leq \frac{\epsilon }{\sqrt L}$, and 
for each $k'>k(\ell)$, 
\begin{align*}
 \frac{N_{k(\ell ) , k'}}{N_{k(-1), k'} } -  \frac{ N_{k(\ell ), k'-1}}{N_{ k(-1), k'-1}}
=\frac{N_{k(-1), k(\ell )} m_{k'}}{N_{k(-1) , k'} N_{k(-1) , k' -1}} 
\leq \frac{\epsilon }{\sqrt L} ,
\end{align*}
implying that one can find $k(\ell +1)>k(\ell )$ such that  $\left\vert  \frac{N_{k(\ell ),k(\ell +1)}}{N_{k(-1), k(\ell +1)}}  - T_{\ell +1} \right\vert \leq \frac{\epsilon }{\sqrt  L}$. 
From this,  the conclusion immediately follows. 
\end{proof}

\subsection{Construction of an adapted code}\label{subsection:0516c2}

Let $\{ \tilde \epsilon _L\} _{L\geq 0}$ be a sequence of positive numbers such that $\lim _{L\to \infty} \tilde \epsilon _L =0$. 
For each $L\geq 1$, let 
$\{ \mathbf t_{L ,j} \} _{j=1}^{J(L)}$ be a finite subset of $A_L$ such that 
 $\{ B _{\tilde \epsilon _L /(L+1) } ( \mathbf t_{L,j }) \} _{j=1}^{J(L)}$ covers $A_L$. 
 (Recall that $B_\epsilon (\mathbf t)$ is the ball of radius $\epsilon$ and center $\mathbf t$.) 
If we write $\mu _{L,j}$ for $\mu _{\mathbf t}$ with $\mathbf t =\mathbf t_{L,j}$, then it follows from Lemma \ref{lem:reset2b} that 
 $\{ B _{\tilde \epsilon _L  } ( \mu_{L,j }) \} _{j=1}^{J(L)}$ covers $\Delta _L(\mathcal J )$.
We consider  a lexicographic  order in  
$\mathbb A= \{ (L ,j ) \} _{L\geq 1, 1\leq j\leq J(L)}$ by 
\[
(L' ,j' ) \leq (L ,j ) \quad \text{ if $L'<L$, or $L'=L$ and $j'\leq j$. }
\]

We define finite increasing sequences  of positive integers  $ \mathbf k_{L, j } =\{ k_{L, j}(\ell ) \} _{\ell =-1}^L$ inductively with respect to $(L, j )\in\mathbb A$. 
Let $\mathbf k_{1,1} =\{ k_{1,1}(\ell )\} _{\ell =-1}^1 $ 
be a finite increasing sequence of  positive integers such that  
\[
\left \vert \overline t (\overline M( \mathbf k_{1,1})) - \mathbf t _{1,1} \right\vert \leq \tilde \epsilon _1.
\] 
We can take such $\mathbf k_{1,1}$ by virtue of Lemma \ref{lem:12}. 
Let $(L,j)\in \mathbb A $, and assume that $\mathbf k_{L', j'} = \{  k_{L', j'} (\ell )\}_{\ell =-1}^{L'}$ is defined for any $(L ', j')\in \mathbb A$ satisfying $(L ', j')<(L,j)$. 
Then we take $\mathbf k_{L,j} =\{ k_{L,j}(\ell )\} _{\ell =-1}^L$ as a finite increasing sequence of  positive integers such that, if we write $(L',j')$ for the  predecessor of $(L,j)$ (i.e.~$L'=L$ and $j'=j-1$, or $L'=L-1$ and $j'=J(L')$, $j=1$), then
\begin{equation}\label{eq:tildeepsilon}
k_{L, j }(-1) = k_{L', j'} (L') \quad \text{and} \quad \frac{2 N_{k_{L, j }(-1 )} }{N_{k_{L, j }(L)} }< \tilde \epsilon _L , 
\end{equation}
and  that 
\begin{equation}\label{eq:tildeepsilon2}
\left\vert \overline t ( \overline M( \mathbf k_{L, j})) - \mathbf t _{L,j} \right \vert \leq \frac{\tilde \epsilon _L}{L+1}. 
\end{equation}
Again, we can take such $\mathbf k_{L, j}$ due to Lemma \ref{lem:12}. 
Finally, let $\{ \ell _k\}_{k\geq 1}$ be a sequence of nonnegative integers associated with $\mathbf k_{L, j}$ for all $(L, j)\in \mathbb A$, and $f$  a $\mathcal C^r$ diffeomorphism with a wandering domain $D$ coded by $\{ \ell _k\} _{k\geq 1}$ over $\{ m_k\} _{k\geq 1}$.

We now complete the proof of  Theorem \ref{thm:0411b}.
Fix  $\tilde L\geq 1$, $\mathbf t\in \Delta _{\tilde L}$ and $\epsilon >0$. 
Let $L\geq \tilde L$ be an integer such that 
$\tilde \epsilon _L<\epsilon /4$ and 
\begin{equation}\label{eq:1008d}
\epsilon _{k_1} + \left\vert \frac{\sum _{k=k_1} ^{k_2}\# I_k}{\sum _{k=k_1} ^{k_2} m_k} -\zeta \right\vert +  \left\vert 1- \frac{\sum _{k=k_1} ^{k_2}(\# I_k +\# \widehat I_k)}{\sum _{k=k_1} ^{k_2} m_k}  \right\vert < \frac{\epsilon}{8}
\end{equation}
for all $k_2>k_1 \geq k_{L,1}(-1)$ (see Remark \ref{rmk:tildeepsilon}). 
Let $1\leq j\leq J(L)$ be an integer such that 
$ d ( \mu_{\mathbf t} , \mu _{L , j} ) \leq \tilde \epsilon _L$ (one can find such $j$ by the construction of $\{ \mathbf t_{L,j} \} _{j=1}^{J(L)}$ and the fact $\Delta _{\tilde L}(\mathcal J ) \subset \Delta _{L}(\mathcal J )$ together with Lemma \ref{lem:reset2b}). 
Then, with the notation 
$\tilde{\mathbf k }
=\mathbf k_{L , j} $ and $\tilde N=N_{k_{L , j}(L)}$,
it follows from Lemma \ref{prop:13}, \eqref{eq:tildeepsilon} and \eqref{eq:1008d} that 
\[
 d \left( \delta _x^{\tilde N} , \mu _{\overline t( \overline M( \tilde{\mathbf k}))} \right) \leq \tilde \epsilon _{L} + \frac{\epsilon}{4} \quad \text{for all $x\in D$},
\]
and from  \eqref{eq:tildeepsilon2} and Lemma \ref{lem:reset2b} that 
\[
 d  \left(   \mu _{\overline t( \overline M( \tilde{\mathbf k}))}, \mu _{L ,j} \right) \leq \tilde \epsilon _{L} . 
\]
Therefore, we get $ d \left( \delta _x^{\tilde N} , \mu _{\mathbf t} \right) <\epsilon $ for all $x\in D$.
Since $\tilde L\geq 1$, $\mathbf t\in \Delta _{\tilde L}$ and $\epsilon >0$ are arbitrary, we conclude that 
$D\subset E (\mathcal J , f)$. This completes the proof of Theorem \ref{thm:0411b}.

\appendix

\section{Proof of \eqref{eq:0604b}}\label{s:is}
  In this appendix
 we give the proof of \eqref{eq:0604b}, by following \cite[Theorem 2.1]{BB}.
Let $X \subset \{ 1,2, \ldots ,m\} ^{\mathbb N}$ be a subshift with the specification property with $m\geq 2$, 
endowed with the metric $d_X$ given by 
 $d_{X}(x ,y ) = \sum _{j=0}^\infty \frac{\vert x_j -y_j \vert }{\beta ^j}$ for $x=(x_0 ,x_1, \ldots ), y=(y_0, y_1, \ldots ) \in X$ with $\beta >1$.
 Let $f: X\to X$ be the left shift operator.
For each $\epsilon >0$, 
we let
$T (\epsilon , X )$ be the maximal cardinality
of a set $\widetilde F$ consisting of periodic orbits of $f$
 such that 
$ d_X(x,y)\geq \epsilon$  for any $O, O' \in \widetilde F$ with $O\neq O'$ and   $ x\in O, \; y\in O'$. 
Then, it easily follows from \cite[Theorem 1.6]{BB} that 
\[
\liminf _{\epsilon \to 0} \frac{\log \log N(\epsilon , \mathcal P_{f} (X)  )}{-\log \epsilon } \geq \liminf _{\epsilon \to 0} \frac{\log T(\epsilon , X )}{-\log \epsilon } .
\]
Therefore, 
\eqref{eq:0604b} follows from
 the following claim:
\begin{equation}\label{lem:0604}
\liminf _{\epsilon \to 0} \frac{\log T(\epsilon , X )}{-\log \epsilon } \geq \mathrm{dim} (X ). 
\end{equation}
We will prove \eqref{lem:0604}.
Given $\epsilon >0$, let $N(\epsilon )$ be a positive integer  such that $\beta ^{-N(\epsilon )}\leq \epsilon < \beta ^{-N(\epsilon )+1}$.
 Then, for any periodic orbitss $O, O'$ of period $N(\epsilon )$ with $O  \neq O'$, 
 we have $\min \{ d_X (x,y ) \mid x\in O, \; y\in O'\}  \geq \beta ^{-N(\epsilon )+1}> \epsilon$. 
 On the other hand, since any subshift is expansive (\cite[Section 16]{DGS1976}), by applying \cite[Theorem 22.7]{DGS1976} for continuous maps with expansiveness and the specification property, we get that
 \[
 h_{\mathrm{top}}(f) = \lim _{n\to\infty} \frac{1}{n}\log \mathrm{Per} _n(f),
 \] 
 where $ h_{\mathrm{top}}(f)$ is the topological entropy of $f$ and $\mathrm{Per} _n(f)$ is the number of periodic points of period $n$.
 Therefore, for any $\delta >0$, by taking $\epsilon $ sufficiently small, we get 
\[
T(\epsilon , X) \geq  \frac{\mathrm{Per } _{N(\epsilon )}(f)}{N(\epsilon )} \geq    e^{h_{\mathrm{top}}(f) (1-\delta ) N(\epsilon )} \geq   e^{-h_{\mathrm{top}}(f) (1-\delta )\log \epsilon / \log \beta }.
\]
So, it follows from  Furstenberg's formula $\mathrm{dim} (X) =  h_{\mathrm{top}}(f) / \log \beta $ for subshifts (\cite{Furstenberg1967}) that
\[
\liminf _{\epsilon \to 0} \frac{\log T(\epsilon , X)}{-\log \epsilon } \geq \frac{(1-\delta )   h_{\mathrm{top}}(f)}{\log \beta} =(1-\delta ) \mathrm{dim} (X).
\]
Since $\delta >0$ is arbitrary, we get the claim.

\section*{Acknowledgments}
We would like to express our  gratitude to P.~Barrientos,    Y.~Cao, Y.~Chung, T.~Persson, 
A.~Raibekas, H.~Takahasi, K.~Yamamoto and D.~Yang 
  for many fruitful discussions
and valuable comments. 
 Furthermore, 
  we are   grateful to P.~Berger and S.~Biebler for their valuable comments on stretched exponential emergences,
  in particular, their  pointing out of a
mistake in an earlier version of the proof of Corollary \ref{cor:0611}.
  Moreover, we express  our great appreciation to  the anonymous reviewers for many  suggestions, 
all of which substantially improved the paper,
especially on the presentation of the paper  and the exposition of pointwise emergence for dynamics with the specification property. 
The first and second authors are  grateful to  the members of
Soochow University  for their kind hospitality
when they visited there.
The  second author is also grateful to  the members of
 Lund university for their warm hospitality. 
This work was partially supported by JSPS KAKENHI
Grant Numbers 21K03332, 19K14575 and 18K03376.

\end{document}